\documentclass{amsart}

    \usepackage[top=1 in,bottom=1in, left=1 in, right=1 in]{geometry}
      \usepackage{amsfonts,amsmath,color,fancybox}

     \newcommand{\comment}[1]{\fbox{ \color{blue} #1 }}

\newcommand{\notready}[1]{\comment{...}}
\renewcommand{\notready}[1]{\comment{From NotReady:} #1}

\renewcommand{\vec}[1]{\mathbf{#1}}

\usepackage{fancybox,graphicx,dsfont,pgf}
\usepackage{amsmath,amsthm,amsfonts,amssymb,fancybox,float}
\usepackage{subfigure,color,wrapfig}

\usepackage{soul}
 \floatstyle{ruled}
\restylefloat{table}
\numberwithin{equation}{section}

\newcommand{\mA}{\mathbf{A}}
\newcommand{\mB}{\mathbf{B}}
\newcommand{\mD}{\mathbf{\Delta}}
\newcommand{\mDD}{\mathbf{D}}
\newcommand{\mC}{\mathbf{C}}
\newcommand{\mF}{\mathbf{F}}
\newcommand{\mG}{\mathbf{G}}

\newcommand{\mU}{\mathbf{U}}
\newcommand{\mV}{\mathbf{V}}

\newcommand{\mZ}{\mathbf{Z}}
\newcommand{\mS}{\mathbf{S}}
\newcommand{\mR}{\mathbf{R}}

\newcommand{\mQ}{\mathbf{Q}}
\newcommand{\mI}{\mathbf{I}}

\newcommand{\topp}[1]{^{(#1)}}

\newcommand{\Var}{{\rm Var}}
\newcommand{\toD}{\Rightarrow}

\newcommand{\ve}{\vec{e}}
\newcommand{\vf}{\vec{f}}
\newcommand{\vg}{\vec{g}}

\newcommand{\vu}{\vec{u}}
\newcommand{\vv}{\vec{v}}
\newcommand{\vx}{\vec{x}}
\newcommand{\vxr}{\vec{x}_{(r)}}

\newcommand{\vp}{\vec{p}}
\newcommand{\vt}{\vec{t}}

\def\RR{{\mathbb R}}

\def\NN{{\mathbb N}}

\def\eps{\varepsilon}
\newcommand{\la}{\lambda}
\newcommand{\La}{\Lambda}

\newcommand{\calM}{\mathcal{M}}
\newcommand{\calN}{\mathcal{N}}

\newcommand{\E}{\mathds{E}}
\renewcommand{\Pr}{\mathds{P}}

\newtheorem{theorem}{Theorem}[section]
\newtheorem{proposition}[theorem]{Proposition}
\newtheorem{lemma}[theorem]{Lemma}

\theoremstyle{definition}
\newtheorem{definition}{Definition}[section]
\newtheorem{assumption}{Assumption}[section]
\theoremstyle{remark}
\newtheorem{remark}{Remark}[section]

\usepackage{amssymb}

\newcommand{\diag}{{\rm diag}}

\newcommand{\RRN}{\left(\mI_N-\tfrac{1}{\la^2}\mC^*\mC\right)^{-1}}
\newcommand{\RRM}{\left(\mI_M-\tfrac{1}{\la^2}\mC\mC^*\right)^{-1}}

\title[Singular values of random matrices]{Singular values of large non-central random matrices}
\author{W\l odek Bryc}
\address
{
W\l odek Bryc, Department of Mathematical Sciences, University of Cincinnati, PO Box 210025, Cincinnati, OH 45221--0025, USA
}
\email{wlodzimierz.bryc@uc.edu}

\author{Jack W. Silverstein}

\address{Jack W. Silverstein,
Department of Mathematics, Box 8205, North Carolina State University, Raleigh, NC 27695-8205, USA
}
\email{jack@ncsu.edu}

 	
\begin{document}
\maketitle
\begin{abstract} We study largest singular values of large random matrices, each with mean of a fixed rank $K$.  Our main result is a limit theorem as the
number of rows
and columns approach infinity, while their ratio approaches a positive
constant. It
provides a decomposition
 of the largest $K$ singular values into the deterministic rate of growth,
 random centered fluctuations given as explicit linear combinations of the entries of the matrix, and a term negligible in probability.
We use this representation to   establish asymptotic normality of the largest singular values for random matrices with means that have block structure.
We also deduce asymptotic normality  for the largest eigenvalues of a random matrix arising in a model of population genetics.
\end{abstract}

\section{Introduction} Finite rank perturbations of random matrices   have been studied by numerous authors,
starting with \cite{Lang64}, \cite{furedi1981eigenvalues}.
This paper can be viewed as an extension of work done in %
\cite{Silverstein:1994}  which describes the limiting
behavior of the largest singular value, $\lambda_1$,  of the $M\times N$ random matrix $\mDD^{(N)}$,
consisting of i.i.d. random variables with common mean $\mu>0$, and $M/N\to c$ as $N\to\infty$.
Immediate results are obtained using known properties on the spectral behavior of centered matrices.  Indeed, express $\mDD\topp N$ in the form
\begin{equation}\mDD^{(N)}=\mC\topp N+\mu\text{\bf 1}_M\text{\bf 1}_N^* \label{tag1.1} \end{equation}
where $\mC\topp N$ is an $M\times N$ matrix containing of i.i.d. mean 0 random variables having variance $\sigma^2$ and finite fourth moment,
and $\text{\bf 1}_k$ is the  $k$ dimensional vector consisting of 1's ($*$ denotes transpose).   When the entries of $\mC\topp N$ come
from the first $M$ rows and $N$ columns of a doubly infinite array of random variables, then it is
known %
(\cite{Yin:1988}) that the largest singular value of $\frac1{\sqrt N}\mC$ converges a.s. to $\sigma(1+\sqrt c)$ as $N\to\infty$.
Noticing that the sole positive singular value of
$\mu\text{\bf 1}_M\text{\bf 1}_N^*$ is $\mu\sqrt{MN}$, and using the fact that (see for example %
\cite{stewart1990matrix})
$$|\lambda_1-\mu\sqrt{MN}|\leq\|\mC\|,$$
($\|\cdot\|$ denoting spectral norm on rectangular matrices) we have that almost surely, for all $N$
$$\lambda_1=\mu\sqrt{MN}+O(\sqrt N).$$
From just considering the sizes of the largest singular values of $\mC\topp N$ and $\mu\text{\bf 1}_M\text{\bf 1}_N^*$ one can take the
 view that $\mDD_N$ is a perturbation of a rank one matrix.

 A result in \cite{Silverstein:1994}  reveals that the difference between $\lambda_1$ and $\mu\sqrt{MN}$ is smaller than $O(\sqrt N)$.   It is shown that

\begin{equation}
  \label{Silverstein-94} \lambda_1=\mu\sqrt{MN}+\frac12\frac{\sigma^2}{\mu}\left(\sqrt{\frac MN}+\sqrt{\frac NM}\right) +
\frac1{\sqrt{MN}}\text{\bf 1}_M\mC\topp N\text{\bf 1}_N^*+\frac1{\sqrt M}Z_N,
\end{equation}
where $\{Z_N\}$ is tight (i.e. stochastically bounded).  Notice that the third term converges in distribution to an $N(0,\sigma^2)$ random variable.

This paper generalizes the result in \cite{Silverstein:1994} by both increasing the rank of the second term on the right of \eqref{tag1.1}
 while maintaining the same singular value
dominance of this
term over the random one,  and relaxing the assumptions on the entries of $\mC$.
 The goal is to cover the setting that is
  motivated by applications to population biology \cite{Patterson2006},
\cite{bryc2013separation}, see also Section \ref{Sect: SNP}.

We use the following   notation. We  write $\mA\in\calM_{M\times N}$ to indicate the dimensions of a real matrix $\mA$, and we denote by $\mA^*$
 its transpose;
 $[\mA]_{r,s}$ denotes the $(r,s)$ entry of matrix $\mA$.  $\mI_d$ is the $d\times d$ identity matrix. We use %
the spectral
norm  $\|\mA\|=\sup_{\{\vx:\|\vx\|=1\}}\|\mA \vx\|$ and
occasionally
the Frobenius norm $\|\mA\|_F=\sqrt{\mbox{tr}\mA\mA^*}$.
Recall that for an $M\times N$ matrix we have
\begin{equation}
  \label{Frobenius}
  \|\mA\|\leq \|\mA\|_F.
\end{equation}
Vectors are denoted by lower case boldface letters like $\vx$ and treated as column matrices so that the Euclidean length is
$\|\vx\|^2=\vx^*\vx$.

Throughout the paper, we use the same letter $C$ to denote various positive non-random constants that do not depend on $N$.
The phrase "... holds for all $N$ large enough" always means that there exists a non-random $N_0$ such that "..." holds for all $N>N_0$.

We fix $K\in\NN$ and a sequence of integers $M=M\topp N$ such that
\begin{equation}
  \label{eq:M/N}
  \lim_{N\to\infty} M/N=c>0.
\end{equation}
  As a finite rank perturbation  we take a sequence of deterministic rank $K$ matrices $\mB=\mB\topp N\in\calM_{M\times N}$ with $K$ largest singular values
  $\rho_1\topp N\geq \rho_2\topp N\geq \dots \geq \rho_K\topp N$. 
  \begin{assumption}
    \label{A1}
  We assume that the limits
  \begin{equation}\label{Assume-rho-lim}
    \gamma_r:=\lim_{N\to\infty}\frac{\rho_r\topp N}{\sqrt{MN}}
  \end{equation}
  exist and are distinct and strictly positive, $\gamma_1>\gamma_2>\dots>\gamma_K>0$.
  \end{assumption}
Our second set of assumptions deals with randomness.
\begin{assumption}
Let $\mC=\mC\topp N$ be an $M\times N$ random matrix with real independent entries $[\mC]_{i,j}=X_{i,j}$. (The distributions of the
entries may differ, and may depend on $N$.) We assume that $\E(X_{i,j})=0$ and that there exists a constant $C$ such that
\begin{equation}\label{Assume-tail}
  \E(X_{i,j}^4)\leq C.
\end{equation}
\end{assumption}
In particular,  the variances
\begin{equation*}
\sigma_{i,j}^2=E(X_{i,j}^2)
\end{equation*}
are uniformly bounded.

We are interested in the asymptotic behavior of the $K$ largest singular values $\la_1\geq\la_2\geq\dots\geq\la_K$ of the  sequence of
noncentered random $M\times N$ matrices
\begin{equation}
  \label{eq:D}
  \mDD=\mDD\topp
N=\mC+\mB.
\end{equation}

Our main result represents each singular value $\la_r$ as a sum of four terms, which represent the rate of growth, random centered fluctuation, deterministic shift, and a term negligible in probability.
To state this result we need additional notation.

The rate of growth is determined by the singular values {$\rho_1\topp N\geq\rho_2\topp N\geq\dots\geq \rho_K\topp N>0$}
 of  the deterministic perturbation matrix $\mB$. For ease of notation, we will occasionally omit superscript $(N)$.
 \newcommand{\mRho}{\mathbb{R}}
 The singular value decomposition of $\mB$ can be written as
\begin{equation}
  \label{SVD(B)}
  \mB=\mF \diag(\rho_1,\dots,\rho_K)\mV^*=\mF\mG^*,
\end{equation}
where $\mF\in\calM_{M\times K}$ has orthonormal columns and $\mG=\mV \diag(\rho_1,\dots,\rho_K)\in\calM_{N\times K}$
 has orthogonal columns of lengths $\rho_1,\dots,\rho_K$.

The random centered fluctuation term for the $r$-th singular value $\la_r$ is \begin{equation}
  \label{eq:Z_r}
  Z_r\topp N=\frac{[\mZ_0]_{r,r}}{ \gamma_r},
\end{equation}
where $\mZ_0$ is a random centered $K\times K$ matrix given by
\begin{equation}
  \label{eq:Z_0}
  \mZ_0=\frac{1}{\sqrt{MN}}\mG^*\mC^*\mF.
\end{equation}
Note that $\mZ_0$  implicitly %
depends  on $N$.

The expression for the constant shift  depends on the variances of the entries of $\mC$. To write the expression we introduce
a diagonal $N\times N$ matrix  $\mD_R=\E(\mC^*\mC)/M$ and a diagonal $M\times M$ matrix $\mD_S=\E(\mC\mC^*)/N$. The diagonal entries of these matrices are
$$\left[\mD_R\right]_{j,j}=\frac{1}{M}\sum_{i=1}^M \sigma_{i,j}^2,
 \quad \left[ \mD_S\right]_{i,i}=\frac{1}{N}\sum_{j=1}^N \sigma_{i,j}^2.$$
Let  $\mathbf{\Sigma}_r=\mathbf{\Sigma}_R\topp N$ and $\mathbf{\Sigma}_S=\mathbf{\Sigma}_S\topp N$ be deterministic $K\times K$ matrices given by
\begin{equation}
  \label{Sigma_R}
  \mathbf{\Sigma}_R=\mG^*\mD_R\mG%
\end{equation}
and
\begin{equation}
  \label{Sigma_S}
  \mathbf{\Sigma}_S=\mF^*\mD_S\mF%
\end{equation}

Define \begin{equation}   \label{mmm}
m_r\topp N=\frac12\left[  \frac {\sqrt{c}}{\gamma_r^3
MN}\mathbf{\Sigma}_R+\frac{1}{ \sqrt{c}\gamma_r}\mathbf{\Sigma}_S \right]_{r,r}.
\end{equation}

\begin{theorem}
  \label{Thm:expansion}
With the above notation,  there exist $\eps_1\topp N\to 0,\dots,\eps_K\topp N\to 0$ in probability such that for $1\leq r\leq K$ we have
\begin{equation}
  \label{eq:la-expand}
  \la_r= \rho_r\topp N+Z_r\topp N+ m_r\topp N  +\eps_r\topp N.
\end{equation}

\end{theorem}
Expression \eqref{eq:la-expand} is less precise than \eqref{Silverstein-94} (where { the negligible term} $ Z_N/\sqrt{M}\to 0$ in probability at known rate  {as $Z_n$ is stochastically bounded}), but it is strong enough  to establish asymptotic normality
under appropriate additional conditions. Such applications {require additional assumptions and} are worked out in Section \ref{Sect:AsymptNorm}.

\begin{remark}\label{Remark-tilde-all}
  In our motivating example in Section \ref{Sect: SNP}, the natural setting is
   factorization of finite rank perturbation as
   \begin{equation}\label{B-from-tilde}
   \mB=\widetilde \mF\widetilde \mG^*=\sum_{s=1}^K\widetilde\vf_s\widetilde\vg_s^*,
   \end{equation} where  $\widetilde \mF\in \calM_{M\times K}$ has orthonormal columns  { $\widetilde\vf_s$},
   but   the columns, $\widetilde\vg_s$, of $\widetilde \mG\in\calM_{N\times K}$ are not necessarily orthogonal.
These matrices are a natural input for the problem so we would like to maintain their roles
   and  recast Theorem \ref{Thm:expansion} in terms of such matrices.
 We introduce
  matrices
   $\widetilde \mR_0=\widetilde \mR_0\topp N\in\calM_{K\times K}$ by
\begin{equation}
  \label{R0}
 \widetilde \mR_0=\widetilde \mG^*\widetilde \mG
\end{equation}
with eigenvalues $\rho_1^2\geq\dots\geq \rho_K^2$ and we denote the corresponding orthonormal eigenvectors by
 $\widetilde\vu_1,\dots,\widetilde \vu_K$.
We claim that  \eqref{eq:Z_r} is
\begin{equation}
  \label{eq:Z_r-b}
  Z_r\topp N=\frac{1}{ \gamma_r}\widetilde\vu_r^*\widetilde\mZ_0\widetilde\vu_r,
\end{equation}
with   %
\begin{equation}
  \label{eq:tildeZ_0}
  \widetilde\mZ_0=\frac{1}{\sqrt{MN}}\widetilde\mG^*\mC^*\widetilde\mF,
\end{equation}
and similarly that
 \eqref{mmm} is
\begin{equation}   \label{mmm-b} m_r\topp N=\frac{1}{2\sqrt{c}\gamma_r}\widetilde\vu_r^* \left(\frac {c}{\gamma_r^2
MN}\widetilde{\mathbf{\Sigma}}_R+\widetilde{\mathbf{\Sigma}}_S\right)\widetilde\vu_r.
\end{equation}
with $\widetilde \mG$ and $\widetilde \mF$ used in expressions \eqref{Sigma_R} and \eqref{Sigma_S} which define
$\widetilde{\mathbf{\Sigma}}_R$ and $\widetilde{\mathbf{\Sigma}}_S$.

Indeed, let $\widetilde \mG^*=\widetilde \mU \diag(\rho_1,\dots,\rho_K)\widetilde \mV^*$ be its singular value decomposition.
If $N$ is large enough so that the singular values are distinct, we may assume that  $\widetilde\vu_1,\dots,\widetilde \vu_K$ are
the columns of $\widetilde \mU$. Then $\widetilde \mV$ is determined uniquely and
 $\mB=\widetilde \mF\widetilde \mU \diag(\rho_1,\dots,\rho_K) \widetilde \mV^*$.
Since by assumption our
singular values are distinct,   with proper alignment we may assume that $\mG=\diag(\rho_1,\dots,\rho_K ) \widetilde \mV^*$ and then necessarily
$\mF=\widetilde \mF\widetilde \mU $  in \eqref{SVD(B)}.
So for  any $\mA\in\calM_{N\times M}$ we have
$$
\widetilde \vu_r^*\widetilde \mG^* \mA \widetilde \mF\widetilde \vu_r = [\mG^*\mA\mF]_{r,r},
$$
which we apply three times with $\mA=\mC^*$, $\mA=\widetilde{\mathbf{\Sigma}}_R$ and $\mA=\widetilde{\mathbf{\Sigma}}_S$.
\end{remark}

\section{Proof of Theorem \ref{Thm:expansion}}

Throughout the proof,
we assume that all our random variables are defined on a single probability space $(\Omega,\mathcal{F},\Pr)$.
 In the initial parts of the proof we will be working on   subsets   $\Omega_N\subset \Omega$ such that $\Pr(\Omega_N)\to 1$.

\subsection{Singular value criterion}

In this section, we fix $r\in\{1,\ldots,K\}$ and let $\la=\la_r$.  For $\| \mC\|^2< \la^2$, matrices $\mI_M-\tfrac{1}{\la^2}\mC\mC^*$ and $\mI_N-\tfrac{1}{\la^2}\mC^*\mC$  are
invertible, so we
consider the following $K\times K$ matrices:
\begin{eqnarray}\label{eq:Z}
  \mZ&=&\frac{1}{\la}\mG^* \RRN\mC^*\mF\,,
\\ \label{eq:S}
  \mS&=& \mF^*\RRM\mF\,,
\\ \label{eq:R}
  \mR&=& \mG^*\RRN\mG\,.
\end{eqnarray}
(These auxiliary random matrices depend on $N$ and $\la$, and are well defined only on a subset of the probability space $\Omega$.
{We will see that these matrices are critical for our subsequent analysis.})
\begin{lemma}
  If $\|\mC\|^2<   \la^2/2$ %
  then
  \begin{equation}\label{eq:det-1}
    \det \begin{bmatrix}
      \mZ-\la \mI_K & \mR \\
      \mS & \mZ^*-\la \mI_K
    \end{bmatrix} =0.
  \end{equation}
\end{lemma}

\begin{proof}  The starting point is the singular value decomposition $\mDD=\mU \Lambda \mV^*$. We choose the $r$-th columns $\vu\in\RR^M$ of $\mU$ and
$\vv\in\RR^N$ of $\mV$ that correspond to the singular value $\la=\la_r$.
  Then \eqref{eq:D} gives
\begin{eqnarray}
\label{eq:mC}  \mC \vv +\sum_{s=1}^K (\vg_s^*\vv) \vf_s &=& \la \vu\,, \\
\label{eq:mC*}  \mC^* \vu +\sum_{s=1}^K (\vf_s^*\vu) \vg_s &=& \la \vv\,.
\end{eqnarray}
where $\vf_s\in\RR^M$ is the $s$-th column of $\mF$ and $\vg_r\in\RR^N$ is the $r$-th column of $\mG$.
Multiplying from the left by $\mC^*$ or $\mC$  we get
\begin{eqnarray*}
 \mC^* \mC \vv +\sum_{s=1}^K (\vg_s^*\vv) \mC^*\vf_s &=& \la \mC^* \vu, \\
\mC \mC^* \vu +\sum_{s=1}^K (\vf_s^*\vu) \mC\vg_s &=& \la \mC\vv.
\end{eqnarray*}
Using \eqref{eq:mC} and \eqref{eq:mC*} to the expressions on the right, we get
\begin{eqnarray*}
 \mC^* \mC \vv +\sum_{s=1}^K (\vg_s^*\vv) \mC^*\vf_s &=& \la^2\vv- \la \sum_{s=1}^K (\vf_s^*\vu) \vg_s\,,  \\
\mC \mC^* \vu +\sum_{s=1}^K (\vf_s^*\vu) \mC\vg_s &=& \la^2\vu - \la \sum_{s=1}^K (\vg_s^*\vv) \vf_s\,.
\end{eqnarray*}
which we rewrite as
\begin{eqnarray*}
 \la^2\vv- \mC^* \mC \vv  &=& \sum_{s=1}^K (\vg_s^*\vv) \mC^*\vf_s+\la \sum_{s=1}^K (\vf_s^*\vu) \vg_s \,, \\
\la^2\vu-\mC \mC^* \vu  &=&   \la \sum_{s=1}^K (\vg_s^*\vv) \vf_s +\sum_{s=1}^K (\vf_s^*\vu) \mC\vg_s\,.
\end{eqnarray*}
This gives
\begin{eqnarray*}
\label{eq:vv=mC}  \vv  &=& \frac1{\la^2}\sum_{s=1}^K (\vg_s^*\vv) \RRN\mC^*\vf_s+\frac1{\la}\sum_{s=1}^K (\vf_s^*\vu) \RRN \vg_s \,, \\
\label{eq:vu=mC}  \vu  &=&  \frac1{\la} \sum_{s=1}^K (\vg_s^*\vv) \RRM \vf_s +\frac1{\la^2}\sum_{s=1}^K (\vf_s^*\vu) \RRM\mC\vg_s\,.
\end{eqnarray*}

Notice that since $\vu,\vv$ are unit vectors, some among the $2K$ the numbers  $(\vf_1^*\vu),\dots,(\vf_K^*\vu)$,
$(\vg_1^*\vv),\dots,(\vg_K^*\vv)$   must be non-zero. Since for $1\leq t\leq K$ we have
\begin{eqnarray*}
  (\vg_t^* \vv)  &=& \frac1{\la^2}\sum_{s=1}^K (\vg_s^*\vv) \vg_t^*\RRN\mC^*\vf_s+\frac1{\la}\sum_{s=1}^K (\vf_s^*\vu) \vg_t^*\RRN\vg_s\,,  \\
  (\vf_t^* \vu)  &=&  \frac1{\la} \sum_{s=1}^K (\vg_s^*\vv) \vf_t^*\RRM \vf_s +\frac1{\la^2}\sum_{s=1}^K (\vf_s^*\vu) \vf_t^*\RRM\mC\vg_s\,,
\end{eqnarray*}
noting that the entries of $\mZ^*$  can be written as
\begin{equation*}
  [\mZ^*]_{s,t}=\frac{1}{\la}\vf_s^*{}(\mI_M-\tfrac{1}{\la^2}\mC\mC^*)^{-1}\mC\vg_t\,.
\end{equation*} we see that the block matrix
$$\begin{bmatrix}
   \frac{1}{\la}\mZ & \frac1{\la}\mR \\ \\
   \frac1\la \mS &  \frac{1}{\la}\mZ^*
\end{bmatrix}$$
has eigenvalue 1. Thus $\det \begin{bmatrix}
   \frac{1}{\la}\mZ-\mI_K & \frac1{\la}\mR \\
   \frac1\la \mS &  \frac{1}{\la}\mZ^*-\mI_K
\end{bmatrix}=0$ which for $\la>0$ is equivalent to \eqref{eq:det-1}.

\end{proof}
\begin{proposition}[Singular value criterion]
    If $\|\mC\|^2<  \la^2/4$   then  $\mS$ is invertible and
   \begin{equation}\label{eq:det-2}
     \det \left((\la \mI_K-\mZ)\mS^{-1}(\la \mI_K-\mZ^*)-\mR\right)=0.
   \end{equation}
\end{proposition}
Similar  equations that involve a $K\times K$ determinant appear in other papers on   rank-$K$ perturbations of random matrices,
compare \cite[Lemma 2.1 and Remark 2.2]{Tao:2013aa}.
Note however that in our case $\la$ enters the equation in a rather complicated way through $\mZ=\mZ(\la),\mR=\mR(\la),\mS=\mS(\la)$. The dependence of these
 matrices on $N$  is also suppressed in our notation.

\begin{proof}
  Note that  if $\|\mC\|^2\leq   \la^2/2$ then  the norms of $ \RRM$ and  $\RRN$ are bounded by 2. Indeed,
$\|\RRM\|\leq \sum_{k=0}^\infty \|(\mC\mC^*)^k\|/\la^{2k}\leq \sum_{k=0}^\infty 1/2^k$.  Since $\RRM=\mI_M+\frac{1}{\la^2}\mC\mC^*\RRM$ and vectors $\vf_1,\dots,\vf_K$ are orthonormal, we get $\mF^*\mF=\mI_K$ and
\begin{equation*}
  \label{eq:S-I}
  \mS-\mI_K=\frac{1}{\la^2} \mF^*{}\mC^*\mC\RRM \mF\,.
\end{equation*}
Since $\|\mF\|=1$ we have
\begin{equation}
  \label{S-I-entries}
\|\mS-\mI_K\| \leq 2 \|\mC\|^2/\la^2.
\end{equation}
We see that if $\|\mC\|^2\leq   \la^2/4$ then
 $\|\mS-\mI_K\|\leq 1/2$,
so the inverse $\mS^{-1}=\sum_{k=0}^\infty (\mI-\mS)^k$ exists. For later %
 {reference} we  also note that
\begin{equation}
  \label{S-inv-norm}
  \|\mS^{-1}\|\leq 2.
\end{equation}

Since
$$
\begin{bmatrix}
       \mR&\mZ -\la \mI_K \\
     \mZ^* -\la \mI_K& \mS
    \end{bmatrix}=
\begin{bmatrix}
      \mZ -\la \mI_K & \mR\\
      \mS & \la \mZ^* -\la \mI_K
    \end{bmatrix}\times \begin{bmatrix}0&\mI_K\\\mI_K & 0\end{bmatrix},
    $$
  we see that \eqref{eq:det-1} is equivalent to
  $$
  \det \begin{bmatrix}
       \mR& \mZ -\la \mI_K  \\
    \mZ^* -\la \mI_K& \mS
    \end{bmatrix}=0.
  $$
Noting that $\la>0$ by assumption, we see that \eqref{eq:det-1} holds and gives \eqref{eq:det-2}, as by  Schur's complement formula
$$\det \begin{bmatrix}
  A & B \\ C &D
\end{bmatrix}=\det D \det (A-BD^{-1}C) .$$
\end{proof}

 \subsection{Equation for $\la_r$}
As was done previously we fix $r\in\{1,\ldots,K\}$, and write  $\la=\la_r$.

The main step in the proof of Theorem \ref{Thm:expansion} is the following expression.
\begin{proposition}\label{Prop:la=} There exists a random sequence $\mathcal{\eps}_c\topp N\to 0$ in probability as $N\to\infty$ such that
\begin{equation}\label{eq:La-epsc}
\la- \rho_r\topp N =  \frac{N}{\la+\ \rho_r\topp N} [\mathbf{\Sigma}_S]_{r,r} + \frac{2\sqrt{MN}}{\la+ \rho_r\topp N}
[\mZ_0]_{r,r}  +  \frac{M}{(\la+ \rho_r\topp N)\la^2} [\mathbf{\Sigma}_R]_{r,r}+
\mathcal{\eps}_c\topp N.
\end{equation}\end{proposition}
The proof of Proposition \ref{Prop:la=} is technical and lengthy.

\subsubsection{Subset $\Omega_N$}
 With $\gamma_{K+1}:=0$, let
   \begin{equation}
     \label{delta}
    \delta:=\min_{1\leq s\leq K} (\gamma_s^2-\gamma_{s+1}^2).
   \end{equation} Assumption \ref{A1}(iii) says that $\delta>0$.

 {In the following the powers of $N$ used are sufficient to conclude our arguments. Sharper bounds would not reduce
 moment assumption \eqref{Assume-tail}, which we use for \cite{latala2005some}
  in the proof of
Lemma \ref{L:P(Omega)}. }
\begin{definition}
  \label{Def-Omega_N} Let  $\Omega_N\subset \Omega$ be such that
\begin{equation}
  \label{eq:|C|}
  \|\mC\|^2\leq N^{5/4}
\end{equation}
and
      \begin{equation}\label{eq:OmegaN2}
\max_{1\leq s\leq K}|\lambda_s^2 - \gamma_s^2 MN|\leq \tfrac\delta4  MN.
      \end{equation}

\end{definition}
We assume that $N$ is large enough  so that $\Omega_N$ is a non-empty set. In fact,  $\Pr(\Omega_N)\to 1$, see Lemma \ref{L:P(Omega)}.

 We note that  \eqref{eq:OmegaN2}  implies that
$c\sqrt{MN} \leq \la_r\leq C\sqrt{MN}$ with $c=\sqrt{\gamma_r^2-\delta/4}\geq \sqrt{\gamma_K^2-\delta/4} >\sqrt{\delta}/2$ and
$C=\sqrt{\gamma_r^2+\delta/4}<2  \gamma_1$. For later reference we state these bounds explicitly:
\begin{equation}\label{eq:la-growth}
\frac{\sqrt{\delta MN}}{2} \leq \la_r\leq 2\ \gamma_1\sqrt{MN}.
\end{equation}
We also note that %
inequalities  \eqref{eq:|C|} and \eqref{eq:la-growth} imply
 that
\begin{equation}   \label{eq:|C|^2}
 \|\mC\|^2<
\frac{\la_K^2}{4}
\end{equation}
 for all  $N$   large enough. (That is, for all $N>N_0$ with nonrandom constant $N_0$.)
Thus matrices
$\mZ$, $\mR$, $\mS$ are well defined on $\Omega_N$ for large enough $N$ and \eqref{eq:det-2} holds. %

\begin{lemma}\label{L:P(Omega)} For  $1\leq r\leq K$ we have $\la_r/\sqrt{MN}\to \gamma_r$ in probability. Furthermore,
  $\Pr(\Omega_N)\to 1$ as $N\to\infty$.
\end{lemma}
\begin{proof}

From \eqref{eq:D} we have $\mDD-\mB=\mC$, so by  Weil-Mirsky theorem \cite[page 204, Theorem 4.11]{stewart1990matrix}, we have a
bound $|\la_r- {\rho_r\topp N}|\leq \|\mC\|$ for the differences between the $K$ largest singular values of $\mB$ and
$\mDD$. From \cite[Theorem 2]{latala2005some} we see that there is a constant $C$ that does not depend on $N$ such that
$E\|\mC\|\leq C (\sqrt{M}+\sqrt{N})$. Thus
$$
\frac{\la_r- \rho_r\topp N}{\sqrt{MN}}\to 0 \mbox{ in probability}.
$$
Since $\rho_r\topp N/\sqrt{MN}\to \gamma_r>0$ this proves the first part of the conclusion.

To prove the second part, we use the fact  that continuous functions preserve convergence in probability, so
$\lambda_r^2 /(MN)\to \gamma_r^2$ in probability for  $1\leq r\leq K$.
Thus
\begin{multline*}
 \Pr(\Omega_N')\leq \Pr(\|\mC\|>N^{5/8})+\sum_{s=1}^K \Pr(|\la_s^2-\gamma_s^2 MN|>\delta MN/4)
\\
\leq
C \frac{\sqrt{M}+\sqrt{N}}{N^{5/8}}+\sum_{s=1}^K \Pr(|\la_s^2/(MN)-\gamma_s^2|>\delta /4)\to 0 \mbox{ as $N\to\infty$}.
\end{multline*}

\end{proof}

\subsubsection{Proof of Proposition \ref{Prop:la=}} In view of \eqref{eq:|C|^2}, equation \eqref{eq:det-2} holds on $\Omega_N$
if $N$ is large enough. It implies that   there is a (random) unit vector $\vx_r\topp N=\vx\in\RR^K$ such that
$$
(\la \mI_K-\mZ)\mS^{-1}(\la \mI_K-\mZ^*)\vx=\mR\vx.$$
{We further choose $\vx$ with non-negative $r$-th component.}
Using diagonal matrix
\begin{equation*}
  \label{R00}
   \mR_0=  \mG^* \mG=\diag(\rho_1^2,\dots,\rho_K^2)
\end{equation*}
we rewrite this as follows.
\begin{equation}
   \label{eq:Jack0}
   (\la^2 \mI_K-\mR_0)\vx=\left(\la^2(\mI_K-\mS^{-1})+\la(\mZ\mS^{-1}+\mS^{-1}\mZ^*)-\mZ\mS^{-1}\mZ^*
+(\mR-\mR_0) \right)\vx.
\end{equation}

We now  rewrite this equation using the (nonrandom) singular values $\rho_1\topp N\geq\rho_2\topp N\geq \dots\geq  \rho_K\topp N> 0$  of $\mB$ and
standard basis $\ve_1,\dots,\ve_K\in\RR^K$.

Suppressing
dependence on $r$ and $N$ in the notation, we insert
$$\vx=(\alpha_1,\ldots,\alpha_K)^*=\sum \alpha_s \ve_s$$
into \eqref{eq:Jack0} and look at the $s$-th component. This shows that
$\la=\la_r$ satisfies the following system of $K$ equations
\begin{multline}\label{eq:Jack1}
(\la^2-\rho_s^2)\alpha_s=\la^2\ve_s^*(\mI_K-\mS^{-1})\vx + \la
\ve_s^*(\mZ\mS^{-1}+\mS^{-1}\mZ^*)\vx + \ve_s^*(\mR-\mR_0)\vx-
\ve_s^*\mZ\mS^{-1}\mZ^*\vx,
\end{multline}
 where $1\leq s \leq K$.
(Recall that this is a system of highly nonlinear equations, as matrices $\mS$, $\mZ$ and $\mR$, and the
  coefficients $\alpha_1,\dots\alpha_K$, depend implicitly on $\la$.)

It turns out that for our choice of $\la=\la_r$ random variable $\alpha_r=\alpha_r\topp N$ is close to its extreme value $1$ while
the other coefficients are asymptotically negligible. Since this only holds on $\Omega_N$ a more precise statement is as
follows.

\begin{lemma}
  \label{L:alpha} There exist  deterministic constants $C$ and $N_0$ such that for all $N>N_0$ and
 $\omega\in\Omega_N$ we have
  \begin{equation}\label{eq:alphar}
1-CN^{-3/8}
\leq \alpha_r \leq 1
  \end{equation}
  and
  \begin{equation}\label{eq:alphas}
  |\alpha_s|\leq (C/\sqrt{K-1}) N^{-3/8} \mbox{ for $s\ne r$}.
  \end{equation}
\end{lemma}
\begin{proof}
  Since $\sum \alpha_s^2=1$, inequality \eqref{eq:alphar} is a consequence of \eqref{eq:alphas}.
Indeed, $\alpha_r^2=1-\sum_{s\ne r}\alpha_s^2\geq 1-C^2N^{-3/4}$ and we  use elementary inequality $\sqrt{1-x}\geq 1-\sqrt{x}$ for
$0\leq x\leq 1$. %

   To prove \eqref{eq:alphas}, we use
  \eqref{eq:Jack1}.
By assumption, $\rho_j ^2/(MN)\to \gamma_j^2$. %
 Using \eqref{delta}, we choose $N$ large enough so that $|\rho_s^2-\gamma_s^2MN|\leq \delta MN/4$. Then, with  $s\ne r$ we
  get
\begin{multline*}
|\la^2-\rho_s^2|= |(\la^2-\gamma_r^2MN) +(\gamma_r^2 - \gamma_s^2)MN+(\gamma_s^2MN-\rho_s^2)|
\\ \geq  |\gamma_r^2-\gamma_s^2|MN-
|\la^2-\gamma_r^2MN|-|\rho_s^2-\gamma_s^2MN|\geq
 ( |\gamma_r^2-\gamma_s^2|-\delta/2)MN\geq \frac{\delta}{2}MN.
\end{multline*}
 From \eqref{eq:la-growth} and \eqref{eq:Jack1} we get %
\begin{equation*} \label{eq:alphas2} \frac{\delta}{2}MN |\alpha_s|\leq 4\gamma_1^2 MN \|\mI_K-\mS^{-1}\|+4\gamma_1\sqrt{MN}\|\mZ\|\,
\|\mS^{-1}\|+\|\mR-\mR_0\|+\|\mZ\|^2\|\mS^{-1}\|.
\end{equation*}
Since $\mI_K-\mS^{-1}=\mS^{-1}(\mS-\mI_K) $  using \eqref{S-inv-norm}  %
we get %
\begin{equation} \label{eq:alphas4}   |\alpha_s|\leq \frac{16\gamma_1^2}{\delta} \|\mS-\mI_K\|+\frac{16\gamma_1}{\delta\sqrt{M N}}\|\mZ\|
+\frac{2}{\delta MN}\|\mR-\mR_0\|+\frac{4}{\delta MN}\|\mZ\|^2.
\end{equation}
We now estimate the
 norms of the $K\times K$ matrices on the right hand side.
From \eqref{S-I-entries} using \eqref{eq:|C|}  and \eqref{eq:la-growth} we get
\begin{equation}   \label{eq:|S-I|+} \|\mS-\mI_K\|\leq \frac{2\|\mC\|^2}{\la^2}\leq
\frac{8N^{5/4}}{\delta MN}=\frac{8}{\delta }N^{1/4}M^{-1}.
\end{equation}

Next, we bound $\|\mZ\|$ using  \eqref{eq:Z}.  Recall that $\|\RRN\|\leq 2$ for large enough $N$.
 From   \eqref{Assume-rho-lim} we have$\|\mG\|\leq 2\sqrt{MN}\gamma_1$, for large enough $N$.
Using this,
\eqref{eq:|C|}
and \eqref{eq:la-growth} we get
 \begin{equation}   \label{eq:|Z|}
 \|\mZ\|\leq \|\mG\|\frac{2}{\la}\|\mC\|\leq \frac{4\sqrt{MN}\gamma_1}{\la} N^{5/8}  \leq
\frac{8\gamma_1}{\sqrt{\delta}}N^{5/8}
 \end{equation}
for large enough $N$.

Next we note that
\begin{equation}\label{eq:R-R0}
   \mR-\mR_0=\frac{1}{\la^2} \mG^* \mC^*\mC\RRN  \mG\,.
\end{equation}

 Thus \eqref{eq:R-R0} with   bounds \eqref{eq:la-growth}, \eqref{eq:|C|}
and the above bound on $\|\mG\|$  give us for large enough $N$
 \begin{equation}   \label{eq:|R-R0|}
 \|\mR-\mR_0\|\leq \frac{8{MN\gamma_1^2}}{\la^2}\|\mC\|^2\leq
   \frac{32 { \gamma_1^2}}{\delta } N^{5/4}.
\end{equation}

 Putting these bounds into \eqref{eq:alphas4} we get
 $$
 |\alpha_s|\leq \frac{128\gamma_1^2}{\delta^2}N^{1/4}M^{-1}+\frac{128 \gamma_1^2}{\delta\sqrt{\delta}}N^{1/8}M^{-1/2} +  \frac{64 { \gamma_1^2}}{\delta^2 } N^{1/4}M^{-1} +
\frac{256\gamma_1^2}{\delta^2} N^{1/4}M^{-1}.
 $$
In view of assumption \eqref{eq:M/N}, this proves \eqref{eq:alphas}.

\end{proof}

The next step is to use Lemma \ref{L:alpha} to rewrite the $r$-th equation in \eqref{eq:Jack1}  to identify the "contributing terms" and the negligible "remainder"  $\mathcal{R}$  which
 is of lower order than $\la$ on $\Omega_N$.  We will accomplish this in several steps,   so we will use the subscripts $a,b,c,\dots$ for bookkeeping purposes.

Define $\vxr$ by
$$\vxr=\sum_{s\neq r}\alpha_s\ve_s.$$

We assume that $N$ is large enough so that the conclusion
of  Lemma \ref{L:alpha} holds and furthermore  that $\alpha_r\geq 1/2$. Notice then that
\begin{equation}\label{xrnorm}\|\vxr\|\leq CN^{-3/8}.\end{equation}  Dividing \eqref{eq:Jack1} with $s=r$ by $\alpha_r$ we get
\begin{equation}\label{square1}
   \la^2-\rho_r^2 =\la^2   \ve_r^*(\mI_K-\mS^{-1})\ve_r + 2 \la
\ve_r^*\mZ\mS^{-1}\ve_r  +   \ve_r^*(\mR-\mR_0)\ve_r- \ve_r^*\mZ\mS^{-1}\mZ^*\ve_r+\sqrt{MN}\eps_a\topp N+\mathcal{R}_a\topp N,
\end{equation}
where
\begin{equation*}
  \eps_a\topp N= \frac1{\alpha_r}\ve_r^*(\mZ_0+\mZ_0^*)\vxr
\end{equation*}
and
\begin{multline*}
\mathcal{R}_a\topp N=\la^2\frac1{\alpha_r}  \ve_r^*(\mI_K-\mS^{-1})\vxr + \la
\frac1{\alpha_r} \ve_r^*(\mZ\mS^{-1}+\mS^{-1}\mZ^*-\frac{\sqrt{MN}}{\la}(\mZ_0+\mZ_0^*))\vxr \\ +\frac1{\alpha_r}
\ve_r^*(\mR-\mR_0)\vxr- \frac1{\alpha_r}    \ve_r^*\mZ\mS^{-1}\mZ^*\vxr\,.
\end{multline*}

Here we slightly simplified the equation noting that since $\mS$ is symmetric, $\ve_r^*\mZ\mS^{-1}\ve_r = \ve_r^*\mS^{-1}\mZ^*\ve_r$.

  Our first task is to  derive a
deterministic bound for $\mathcal{R}_a\topp N$ on $\Omega_N$.
\begin{lemma}
  \label{L:calRa}
  There exist non-random constants $C$ and $N_0$ such that on $\Omega_N$ for $N>N_0$ we have
  \begin{equation*}
    |\mathcal{R}_a\topp N|\leq C N^{7/8}.
  \end{equation*}
\end{lemma}
\begin{proof} The constant will be given by a complicated expression that will appear at the end of the proof. Within the
proof, $C$ denotes the constant from Lemma  \ref{L:alpha}.

Notice that
\begin{equation*}\label{eq:Z-Z0}
  \mZ- \frac{\sqrt{MN}}{\la}\mZ_0= \frac{1}{\la^3}\mG^*\mC^*\mC\RRN \mC^*\mF\,,
\end{equation*}
so for large enough $N$ %
 we get
\begin{equation} \label{eq:|Z-Z0|}
   \|\mZ- \frac{\sqrt{MN}}{\la}\mZ_0\| \leq \frac{4}{\la^3}\sqrt{MN}\gamma_1\|\mC\|^3 \leq
   \frac{32\gamma_1}{\delta^{3/2}}N^{7/8}M^{-1}.
\end{equation}

 Recall \eqref{S-inv-norm}
and recall that $N$ is large enough so that $\alpha_r>1/2$.
Using \eqref{xrnorm} and writing $\mZ\mS^{-1}=\mZ\mS^{-1}(\mS-\mI_K)+\mZ$ we get
\begin{multline*}
    |\mathcal{R}_a\topp N|\leq
4C\la^2 N^{-3/8}\|\mS-\mI_K\|+ 8C \la N^{-3/8}\|\mZ\|\|\mS-\mI_K\|+ 4C \la N^{-3/8}\|\mZ- \frac{\sqrt{MN}}{\la}\mZ_0\|
\\  +  2CN^{-3/8}\|\mR-\mR_0\|+2CN^{-3/8}\|\mZ\|^2
\\
\leq
\frac{128 C\gamma_1^2}{\delta}N^{7/8}+\frac{1024 C\gamma_1^2}{ \delta^{3/2}}M^{-1/2}N+\frac{256 C\gamma_1^2}{ \delta^{3/2}}M^{-1/2}N^{}
\\+\frac{64C\gamma_1^2}{\delta}N^{7/8}+\frac{128C\gamma_1^2}{\delta}N^{7/8}.
\end{multline*}
(Here we used \eqref{eq:|S-I|+}, \eqref{eq:|Z|}, then  \eqref{eq:|Z-Z0|}, \eqref{eq:|R-R0|} and \eqref{eq:|Z|} again.)
 This concludes the proof.

\end{proof}
\begin{lemma}
  \label{L:eps_a} For every $\eta>0$, we have
  \begin{equation*}
    \label{eq:eps_a}
   \lim_{N\to\infty} \Pr\left(\left\{\left|\eps_a\topp N\right|>\eta\right\}\cap \Omega_N\right)=0.
  \end{equation*}
\end{lemma}
\begin{proof} We first verify  that each entry of the matrix
$
N^{-3/8}\mZ_0
$,
which is well defined on $\Omega$, converges in probability to 0. To do so, we bound the second moment of
random variable $\xi= \vf_r^*\mC\vg_s$. Since the entries of $\mC$ are independent and centered random variables,
\begin{equation}
  \label{eq:xi}
  \E\xi^2=\sum_{i=1}^M\sum_{j=1}^N [\vf_r]_i^2 \sigma_{i,j}^2 [\vg_s]_j^2\leq \sup_{i,j}\sigma_{i,j}^2 \|\vg_r\|^2\leq C MN.
\end{equation}

Thus, see \eqref{eq:Z_0}, each entry of matrix $\mZ_0$ has bounded second moment, so
$$
\zeta_N:= N^{-3/8}\max_s|\ve_r^*(\mZ_0+\mZ_0^*)\ve_s|\to 0 \mbox{ in probability}.
$$
To end the proof  we note that by \eqref{eq:alphas}, for large enough $N$ we have $\left|\eps_a\topp N\right|\leq 2CK \zeta_N$ on $\Omega_N$, so
$$
\Pr\left(\left\{\left|\eps_a\topp N\right|>\eta\right\}\cap \Omega_N\right)
\leq \Pr\left(\{|\zeta_N|>\tfrac{\eta}{2CK}\}\cap\Omega_N \right)\leq \Pr\left(|\zeta_N|>\tfrac{\eta}{CK}\right)\to 0.
$$
 \end{proof}

 Using the identity \begin{equation*}
  \label{eq:S-inv-expand}
   \mI-\mS^{-1}=(\mS-\mI)-(\mS-\mI)^2\mS^{-1}
\end{equation*}
 to the first term   we rewrite \eqref{square1} as
\begin{equation}\label{square2}
   \la^2-(\rho_r\topp N)^2 =\la^2   \ve_r^*(\mS-\mI_K)\ve_r +2 \la
\ve_r^*\mZ\ve_r  +   \ve_r^*(\mR-\mR_0)\ve_r-\ve_r^*\mZ\mZ^*\ve_r+
\mathcal{R}_b\topp N+\sqrt{MN}\eps_a\topp N+\mathcal{R}_a\topp N,
\end{equation}
with
\begin{equation*}
\mathcal{R}_b\topp N= -\la^2   \ve_r^*(\mS-\mI_K)^2\mS^{-1}\ve_r+ 2\la \ve_r^*\mZ(\mI_K-\mS)\mS^{-1}\ve_r-
\ve_r^*\mZ(\mI_K-\mS)\mS^{-1}\mZ^*\ve_r\,.
\end{equation*}

\begin{lemma}
  \label{L:calRb}
  There exist non-random constants $C$ and $N_0$ such that on $\Omega_N$ for $N>N_0$ we have
  \begin{equation*}
    |\mathcal{R}_b\topp N|\leq C N^{7/8}.
  \end{equation*}
\end{lemma}
\begin{proof}
   Using \eqref{eq:la-growth} and previous norm estimates \eqref{eq:|S-I|+} and \eqref{eq:|Z|}, we get
\begin{multline*}
|\mathcal{R}_b\topp N|\leq 2\la^2 \|\mS-\mI_K\|^2+4\la\|\mZ\|\|\mS-\mI_K\| +2 \|\mZ\|^2 \|\mS-\mI_K\|\\
 \leq
\frac{512 \gamma_1^2}{\delta^2 } N^{3/2}M^{-1}+ \frac{512 \gamma_1^2}{\delta^{3/2}} N^{11/8}M^{-1/2}
+ \frac{1024 \gamma_1^2}{\delta^2}N^{3/2}M^{-1}.
\end{multline*}
This ends the proof.
\end{proof}
Define $K\times K$ random matrices %
 \begin{equation}
  \label{R1}
  \mR_1=\mG^*\mC^*\mC \mG
\end{equation}
and
  \begin{equation}
  \label{S1}
  \mS_1=\mF^*\mC\mC^* \mF\,.
\end{equation}
Recall that  $\E(\mR_1)=M \mathbf{\Sigma}_R$  and   $\E(\mS_1)=N\mathbf{\Sigma}_S$, see \eqref{Sigma_R} and \eqref{Sigma_S}.

\begin{lemma}\label{L:ZS-norm-bounds}
 There exist non-random constants  $C$ and $N_0$ such that on $\Omega_N$ for $N>N_0$ we have

\begin{equation*}
  \|\mS-\mI_K -\frac{1}{\la^2}\mS_1\|\leq C N^{-3/2}
\end{equation*}
and
\begin{equation*}
   \|\mR-\mR_0-\tfrac{1}{\la^2}\mR_1\|\leq C N^{1/2}.
\end{equation*}

\end{lemma}
\begin{proof}
Notice that %
\begin{equation*}
    \mS-\mI_K-\frac{1}{\la^2}\mS_1
  =\frac{1}{\la^4}\mF^*(\mC\mC^*)^2(\mI_M-\tfrac{1}{\la^2}\mC\mC^*)^{-1}\mF.
\end{equation*}
For large enough $N$ (so that \eqref{eq:|C|} and \eqref{eq:la-growth} hold), this gives %
$$\|\mS-\mI_K-\frac{1}{\la^2}\mS_1\|\leq
\frac{2}{\la^4}\|\mC\|^4 \leq \frac{32}{\delta^2}M^{-2}N^{1/2}.
$$

Similarly, since %
\begin{equation*}
   \mR-\mR_0-\tfrac{1}{\la^2}\mR_1
   =\frac{1}{\la^4}  \mG^*  (\mC^*\mC)^2\RRN  \mG
\end{equation*}
for large enough $N$, we get %
$$
\| \mR-\mR_0-\tfrac{1}{\la^2}\mR_1\|\leq   \frac{2}{\la^4}\|\mG\|^2 \|\mC\|^4 \leq
\frac{128\gamma_1^2}{\delta^2}M^{-1}N^{3/2}.
$$
\end{proof}
 We now rewrite \eqref{square2} as follows.
 \begin{multline*}\label{square-c}
   \la^2-(\rho_r\topp N)^2 \\= \ve_r^*(\mS_1)\ve_r +2 \sqrt{MN}
\ve_r^*\mZ_0\ve_r  +   \frac{1}{\la^2}\ve_r^*(\mR_1)\ve_r-\frac{MN}{\la^2}\ve_r^*\mZ_0\mZ_0^*\ve_r+
\mathcal{R}_c\topp N+\mathcal{R}_b\topp N+\sqrt{MN}\eps_a\topp N+\mathcal{R}_a\topp N,
\end{multline*}
 where
\begin{multline*}
  \mathcal{R}_c\topp N  =\ve_r^*\left(\la^2(\mS-\mI_K-\frac1{\la^2}\mS_1)+ 2(\la\mZ-\sqrt{MN}\mZ_0)+  (\mR-\mR_0-\frac1{\la^2}\mR_1)+(\frac{MN}{\la^2}\mZ_0\mZ_0^*-\mZ\mZ^*)\right)\ve_r\,.
\end{multline*}
\begin{lemma}\label{L:calRc}
   There exist non-random constants $C$ and $N_0$ such that on $\Omega_N$ for $N>N_0$ we have
   $$
   |\mathcal{R}_c\topp N |\leq C N^{7/8}.
   $$
\end{lemma}
\begin{proof} %
As in the proof of Lemma \ref{L:calRa},
the final constant $C$ can  be read out from the bound at the end of the proof. In the proof, $C$ is a constant from Lemma \ref{L:ZS-norm-bounds}.
By the triangle inequality, Lemma \ref{L:ZS-norm-bounds}, \eqref{eq:|Z-Z0|} and \eqref{eq:|Z|}, we have %
  \begin{multline*}
     |\mathcal{R}_c\topp N |\leq \la^2\|\mS-\mI_K-\frac1{\la^2}\mS_1\|+2\| \la\mZ-\sqrt{MN}\mZ_0\|\\+\|\mR-\mR_0-\frac1{\la^2}\mR_1\|+
 \|\mZ-\frac{\sqrt{MN}}{\la}\mZ_0\|(\|\mZ\|+\frac{\sqrt{MN}}{\la}\|\mZ_0\|)
   \\
   \leq 4 C \gamma_1^2M N^{-1/2}+  \frac{128\gamma_1^2}{\delta^{3/2}}M^{-1/2}N^{11/8}+ C N^{1/2}+
   \frac{256 \gamma_1^2}{\delta^2}M^{-1}N^{3/2}+\frac{128\gamma_1^2}{\delta^{3/2}} M^{-1}N^{3/2}.
  \end{multline*}
   (Here we used the bound $\|\mZ_0\|\leq 2\gamma_1  N^{5/8}$, which is derived similarly to \eqref{eq:|Z|}.)
\end{proof}

The following holds on $\Omega$. (Recall that expressions \eqref{S1},  \eqref{eq:Z_0} and \eqref{R1} are well defined  on  $\Omega$.)

\begin{proposition} There exists a random sequence $\mathcal{\eps}_b\topp N\to 0$ in probability as $N\to\infty$ such that
\begin{equation}\label{eq:La-epsb}
   \la-{\rho_r\topp N} =  \frac{1}{\la+{\rho_r\topp N}} \ve_r^*(\mS_1)\ve_r + \frac{2\sqrt{MN}}{\la+{\rho_r\topp N}}
\ve_r^*\mZ_0\ve_r  +  \frac{1}{(\la+{\rho_r\topp N})\la^2} \ve_r^*(\mR_1)\ve_r+
\mathcal{\eps}_b\topp N.
\end{equation}\end{proposition}

\begin{proof}
Let
$$
\mathcal{\eps}_b\topp N =\begin{cases}
   \frac{1}{\la+{\rho_r\topp N}}\left(-\frac{MN}{\la^2}\ve_r^*\mZ_0\mZ_0^*\ve_r+
\mathcal{R}_c\topp N+\mathcal{R}_b\topp N+\sqrt{MN}\eps_a\topp N+\mathcal{R}_a\topp N\right) & \mbox{ on $\Omega_N$},\\
  \la-{\rho_r\topp N} -  \frac{1}{\la+{\rho_r\topp N}} \ve_r^*(\mS_1)\ve_r - \frac{2\la}{\la+{\rho_r\topp N}}
\ve_r^* \mZ_0 \ve_r  -  \frac{1}{(\la+{\rho_r\topp N})\la^2} \ve_r^*(\mR_1)\ve_r & \mbox{ otherwise}.
\end{cases}
$$
By Lemma \ref{L:P(Omega)} we have $\Pr(\Omega_N')\to 0$, so it is enough to show that given $\eta>0$ we have $\Pr(\{|\mathcal{\eps}_b\topp N|>5\eta\}\cap \Omega_N)\to 0$ as $N\to\infty$.
Since the event $|\xi_1+\dots+\xi_5|>5\eta$ is included in the union of events $|\xi_1|>\eta, \dots ,|\xi_5|>\eta$, in view of Lemmas \ref{L:calRa},  \ref{L:calRb}, \ref{L:calRc}, and
Lemma \ref{L:eps_a}
  (recalling that expressions  $\sqrt{MN}/(\la+{\rho_r\topp N})$ and $MN/\la^2$ are bounded by a non-random constant on $\Omega_N$, see \eqref{eq:la-growth})
we only need to verify that
$$\Pr\left( \left\{\frac{\ve_r^*\mZ_0\mZ_0^*\ve_r}{\la+{\rho_r\topp N}}>\eta\right\} \cap \Omega_N\right)\leq
\Pr\left(\frac{\ve_r^*\mZ_0\mZ_0^*\ve_r}{{\rho_r\topp N}}>\eta\right)
\to 0 \mbox{ as $N\to\infty$} .$$
Since for large enough $N$, we have $\rho_r^2\geq \delta MN$, convergence follows from %
\begin{equation}
  \label{Z0-stoch-bounded}
  \E\ve_r^*\mZ_0\mZ_0^*\ve_r\leq \E \|\mZ_0\|^2
\leq \E \|\mZ_0\|_F^2\leq K^2 C,
\end{equation}
where $C$ is a constant from \eqref{eq:xi}.

\end{proof}

\begin{proof}[Proof of Proposition \ref{Prop:la=}] Recall that $\E\mS_1=N\mathbf{\Sigma}_S$ and $\E\mR_1=M\mathbf{\Sigma}_R$. So
expression \eqref{eq:La-epsc} differs from \eqref{eq:La-epsb} only by two terms:
  $$\frac{1}{\la+ {\rho_r\topp N}} \ve_r^*(\mS_1-\E\mS_1)\ve_r$$ and
  $$\frac{1}{(\la+ {\rho_r\topp N})\la^2} \ve_r^*(\mR_1-\E\mR_1)\ve_r.$$
 Since $\rho_r\topp N/\sqrt{MN}\to \gamma_r>0$ and by Lemma \ref{L:P(Omega)} we have $\la/\sqrt{MN}\to  {\gamma_r}$ in probability,
 to end the proof we show that
 $\frac{1}{N}\|\mS_1-\E\mS_1\|_F\to 0$ and $\frac{1}{N^3}\|\mR_1-\E\mR_1\|_F\to 0$ in probability. To do so, we bound  the second moments of the entries of the matrices.
 Recalling \eqref{S1}, we have %
 \begin{multline*}
     f_r^*(\mC\mC^*-\E(\mC\mC^*))f_s=\sum_{k=1}^N \sum_{i\ne j} [\vf_r]_iX_{i,k} X_{j,k}[\vf_s]_j+
     \sum_{k=1}^N\sum_{i=1}^M[\vf_r]_i(X_{i,k}^2-\sigma_{i,k}^2)[\vf_s]_i
     = A_N+B_N \mbox{ (say)}.
 \end{multline*}
By independence, we have
$$
\E (A_N^2)=\sum_{k=1}^N \sum_{i\ne j} [\vf_r]_i^2[\vf_s]_j^2\sigma_{i,k}^2\sigma_{j,k}^2
\leq C \sum_{k=1}^N \sum_{i=1}^N\sum_{j=1}^N [\vf_r]_i^2[\vf_s]_j^2
\leq CN.
$$
Next,
$$\E(B_N^2)=\sum_{k=1}^N\sum_{i=1}^M [\vf_r]_i^2[\vf_s]_i^2\E(X_{i,k}^2-\sigma_{i,k}^2)^2\leq
NC \sqrt{\sum_{i=1}^M [\vf_r]_i^4}\sqrt{ \sum_{i=1}^M [\vf_s]_i^4}\leq CN.$$
This shows that (with a different $C$) we have  $\E \left|f_r^*(\mC\mC^*-\E(\mC\mC^*)f_s)\right|^2  \leq CN$ and hence
$\frac{1}{N}\|\mS_1-\E\mS_1\|_F\to 0$ in mean square and in probability.

Similarly, recalling \eqref{R1} we have %
\begin{multline*}
  \vg_r^*(\mC^*\mC-\E(\mC^*\mC))\vg_s=\sum_{k=1}^M\sum_{i\ne j} [\vg_r]_i [\vg_s]_j X_{k,i}X_{k,j}+
  \sum_{k=1}^M\sum_{i=1}^N [\vg_r]_i [\vg_s]_i (X_{k,i}^2-\sigma_{k,i}^2)
  =\widetilde A_N+\widetilde B_N \mbox{ (say)}.
\end{multline*}
Using independence of entries again, we get %
$$
\E(\widetilde A_N^2)=\sum_{k=1}^M\sum_{i\ne j} [\vg_r]_i^2 [\vg_s]_j^2 \sigma^2_{k,i}\sigma^2_{k,j}
\leq C M \|\vg_s\|^2\|\vg_r\|^2\leq 2 C\gamma_1^4 M^3N^2
$$
for large enough $N$.
Similarly,
$$
\E(\widetilde B_N^2)=\sum_{k=1}^M\sum_{i=1}^N [\vg_r]_i^2[\vg_s]_i^2\E(X_{k,i}^2-\sigma^2_{k,i})^2
\leq CM \|\vg_r\|^2\|\vg_s\|^2\leq 2C\gamma_1^4 M^3N^2.
$$
This shows that (with a different $C$) we have  $ \E \left|\vg_r^*(\mC^*\mC-\E(\mC^*\mC))\vg_s\right|^2 \leq CN^5$ and hence
$\frac{1}{N^3}\|\mR_1-\E\mR_1\|_F\to 0$ in mean square and in probability.
\end{proof}
\subsubsection{Conclusion of proof of Theorem \ref{Thm:expansion}}

Theorem  \ref{Thm:expansion} is essentially a combination of \eqref{eq:La-epsc}, and convergence in probability from Lemma
\ref{L:P(Omega)}.
\begin{proof}[Proof of Theorem \ref{Thm:expansion}]
We need to do a couple more approximations to the right hand side of  \eqref{eq:La-epsc}.
Indeed,
we see that
$$\frac{N}{\la+ {\rho_r\topp N}}=\sqrt{N/M}\frac{\sqrt{MN}}{\la+ {\rho_r\topp N}}\to \frac{\sqrt{c}}{2{\gamma_r}},
 $$
  $$\frac{2\sqrt{MN}}{\la+ \rho_r\topp N} \to \frac{1}{\gamma_1},$$
$$\frac{M}{(\la+ \rho_r\topp N)\la^2}   \sim
 \frac{ \sqrt{M/N}}{2{\gamma_r}}\frac{1}{\gamma_r^2 MN}\sim \frac{\sqrt{c}}{2 \gamma_r^3 MN.}
 $$
 To conclude the proof, we note that sequences $\{[\mathbf{\Sigma}_S\topp N]_{r,r}\}_N$,
 $\{[\mathbf{\Sigma}_R\topp N]_{r,r}/(MN)\}$ are bounded and $\{[\mZ_0\topp N]_{r,r}\}$ is stochastically bounded
  by \eqref{eq:xi}.

\end{proof}
\begin{remark}\label{Remark-Q} Recall \eqref{B-from-tilde} and \eqref{R0} from Remark \ref{Remark-tilde-all}.
Examples in Section \ref{Sect: SNP}   have the additional property that
      \begin{equation}
    \label{Assume-q-lim}
   \tfrac1{MN}\widetilde \mR_0\topp N \to \mQ.
  \end{equation}
  Under Assumption \ref{A1}, the eigenvalues of $\mQ$ are $\gamma_1^2>\gamma_2^2>\dots>\gamma_K^2>0$.
  Denoting by $\vv_r$ the corresponding orthonormal eigenvectors,
   we may assume that the first non-zero component of $\vv_r$ is positive.
  After choosing the appropriate sign, without loss of generality  we may assume that the same component of $\widetilde \vu_r\topp N$
 is non-negative for all $N$. Since by assumption eigenspaces of $\widetilde \mR_0$ are one-dimensional for large enough $N$, we have
 \begin{equation*}\label{conv-of-u}
   \widetilde \vu_r\topp {N}\to\vv_r \mbox{ as $N\to\infty$}.
 \end{equation*}
We claim that in
  \eqref{eq:Z_r-b} and in  \eqref{mmm-b} we can replace vectors $\widetilde \vu_r$ by the corresponding eigenvectors $\vv_r$ of $\mQ$.   Indeed, as in the proof of Theorem \ref{Thm:expansion}
 the entries of the  $K\times K$ matrices $\widetilde{\mathbf{\Sigma}}_R/(MN)$ and $\widetilde{\mathbf{\Sigma}}_S$ are bounded as $N\to\infty$.  Also, we note that each entry $\xi= \widetilde\vf_r^*\mC\widetilde\vg_s$ of matrix $\widetilde \mZ_0$  is stochastically bounded
due to a uniform bound on the second moment:
\begin{equation*}
  \label{eq:xi2}
  \E\xi^2=\sum_{i=1}^M\sum_{j=1}^N [\vf_r]_i^2 \sigma_{i,j}^2 [\vg_s]_j^2\leq \sup_{i,j}\sigma_{i,j}^2 \|\vg_r\|^2\leq C MN.
\end{equation*}
(Compare \eqref{eq:xi}.)  This allows for the replacement of the $\widetilde \vu_r$ with the $\vv_r$.

\end{remark}

\section{Asymptotic normality of singular values}\label{Sect:AsymptNorm}
In this section we apply Theorem \ref{Thm:expansion} to deduce asymptotic normality. To reduce technicalities involved, we begin with the simplest case of
mean with rank 1.  An example with mean of rank 2 is worked out in Section \ref{Sec:K=2}. A more involved application to population biology appears in Section \ref{Sect: SNP}.
We use simulations to illustrate these results.

\subsection{Rank 1 perturbation}
The following is closely related to  \cite[Theorem 1.3]{Silverstein:1994} that was mentioned in the introduction.

\begin{proposition}
Fix an infinite sequence $\{\mu_j\}$ such that the limit %
 $\gamma^2=\lim_{N\to\infty} \frac{1}{N}\sum_{j=1}^N\mu_j^2$ exists {and is strictly positive}.
Consider the case $K=1$, and assume that entries of random matrix $\mDD\in\calM_{M\times N}$ are independent,
with the same mean $\mu_j$ in the $j$-th column, the same variance $\sigma^2$,
and uniformly bounded fourth moments. For the largest singular  value $\la$ of $\mDD$, we have
$\la- \sqrt{M} \left(\sum_{j=1}^N\mu_j^2\right)^{1/2}\toD Z$ where $Z$ is normal with mean
 $\frac{\sigma^2}{2{\gamma}}(\sqrt{c}+1/\sqrt{c})$ and variance $\sigma^2$.
  {(Here $\toD$ denotes convergence in distribution.)}
\end{proposition}
\begin{proof}
In this setting %
$\mB=\vf \vg^*$
with  $\vf=M^{-1/2}[1,\dots,1]^*$, $\vg= \sqrt{M}[\mu_1,\dots,\mu_N]^*$. We get  $\rho_1^2= M\sum_{j=1}^N\mu_j^2$,
 $\gamma_1=\gamma$, $\mathbf{\Sigma}_R=\sigma^2 M\sum_{j=1}^N\mu_j^2$, $\mathbf{\Sigma}_S=\sigma^2$, and
$$\mZ_0=\frac{1}{\sqrt{MN}}\sum_{i=1}^M\sum_{j=1}^N X_{i,j}\mu_j, \mbox{ so \eqref{eq:Z_r} gives } Z_1= \frac{1}{\sqrt{MN}\gamma}\sum_{i=1}^M\sum_{j=1}^N X_{i,j}\mu_j.$$
Thus, %
the largest singular value of $\mDD$ can be written as
$$
\la= \sqrt{M} \left(\sum_{j=1}^N\mu_j^2\right)^{1/2}+\frac{\sigma^2(\sqrt{c}+1/\sqrt{c})}{2{\gamma}}   +
\frac{1}{\sqrt{MN}\gamma}\sum_{i=1}^M\sum_{j=1}^N X_{i,j}\mu_j+\eps\topp N,
$$
where $\eps\topp N\to 0$ in probability.
We have
$$\Var(Z_1)=\frac{\sigma^2}{\gamma^2}\frac1N\sum_{j=1}^N\mu_j^2\to \sigma^2$$
and the sum of the fourth moments of the terms in $Z_1$ is
$$\frac{1}{M^2N^2\gamma^4}\sum_{i=1}^M\sum_{j=1}^N \mu_j^4 \E X_{i,j}^4 \leq \frac{C}M   \left(\frac1N\sum_{j=1}^N\mu_j^2\right)^2 \to 0.$$
So $Z_1$ is asymptotically normal by Lyapunov's theorem  \cite[Theorem 27.3]{Billingsley}.
\end{proof}
\subsection{Block matrices}\label{Sec:K=2} %
Consider $(2M)\times(2N)$ block matrices
$$\mDD=\begin{bmatrix}
  \mA_1 & \mA_2 \\
  \mA_3 & \mA_4
\end{bmatrix} ,$$
where $\mA_1,\dots,\mA_4$ are independent random $M\times N$ matrices.   We assume that the entries of
 $\mA_j$ are independent real identically distributed random variables with mean $\mu_j$, variance $\sigma_j^2$,  and with finite fourth moment.
Then from \eqref{B-from-tilde}, $\mB=\E(\mDD)=\widetilde\vf_1\widetilde\vg_1^*+\widetilde\vf_2\widetilde\vg_2^*$; it is of rank $K=2$ with orthonormal
$$
[\widetilde\vf_1]_i=\begin{cases}
  1/\sqrt{M} & \mbox{ for $1\leq i\leq M$} \\
  0 & \mbox{ for $M+1\leq i\leq 2M$}
\end{cases}
$$

$$
[\widetilde\vf_2]_i=\begin{cases}
0 & \mbox{ for $1\leq i\leq M$} \\
  1/\sqrt{M}  & \mbox{ for $M+1\leq i\leq 2M$}
\end{cases}
$$
and with
$$
[\widetilde\vg_1]_j=\begin{cases}
 \sqrt{M}\mu_1  &    \mbox{ for $1\leq j\leq N$} \\
  \sqrt{M}\mu_2  \hspace{2cm}  & \mbox{ for $N+1\leq j\leq 2N$}
\end{cases}
$$
$$
[\widetilde\vg_2]_j=\begin{cases}
  \sqrt{M}\mu_3 &    \mbox{ for $1\leq j\leq N$} \\
 \sqrt{M}\mu_4  \hspace{2cm} & \mbox{ for $N+1\leq j\leq 2N$}
\end{cases}
$$
So $\widetilde\mG=\sqrt{M}\begin{bmatrix}
  \mu_1 & \mu_3 \\
  \mu_1 & \mu_3 \\
  \vdots & \\
\mu_1 & \mu_3 \\
\mu_2 & \mu_4 \\
  \vdots & \\
\mu_2 & \mu_4
\end{bmatrix}$ and $\widetilde\mR_0=\widetilde\mG^*\widetilde\mG=\begin{bmatrix}
  \|\widetilde\vg_1\|^2 & \widetilde\vg_1^*\widetilde\vg_2 \\
  \widetilde\vg_1^*\widetilde\vg_2 & \|\widetilde\vg_2\|^2
\end{bmatrix}=MN\begin{bmatrix}
\mu_1^2+\mu_2^2 &  \mu_1\mu_3+\mu_2\mu_4\\
\mu_1\mu_3+\mu_2\mu_4 & \mu_3^2+\mu_4^2
\end{bmatrix}$.

Denote by $\la_1\geq \la_2$ the largest singular values of $\mDD$.
\begin{proposition}
{  Suppose} $\widetilde\vg_1$ and $\widetilde\vg_2$ are linearly independent and either  $\widetilde\vg_1^*\widetilde\vg_2\ne0$, or if
   $\widetilde\vg_1^*\widetilde\vg_2=0$ then $\|\widetilde\vg_1\|\ne\|\widetilde\vg_2\|$. Then
 there exist constants $c_1>c_2$ such that
\begin{equation}
  \label{Warm-up-CLT}
  (\la_1-c_1\sqrt{MN}, \la_2-c_2\sqrt{MN})\toD (Z_1,Z_2),
\end{equation}
  where $(Z_1,Z_2)$ is a (noncentered) bivariate normal random variable.

\end{proposition}
\begin{proof}
 To use Theorem \ref{Thm:expansion}, we first verify that Assumption \ref{A1} holds. We have
$$
\widetilde\mR_0=4MN \mQ, \quad \mbox{ where }  \mQ=\frac14\begin{bmatrix}
\mu_1^2+\mu_2^2 &  \mu_1\mu_3+\mu_2\mu_4\\
\mu_1\mu_3+\mu_2\mu_4 & \mu_3^2+\mu_4^2
\end{bmatrix}.
$$

Noting that $\det(\mQ)=\left(\mu _2 \mu _3-\mu _1 \mu _4\right)^2/16$, we see that $\gamma_1\geq\gamma_2>0$ provided that $\det\begin{bmatrix}
\mu_1 & \mu_2 \\
\mu_3 & \mu_4\end{bmatrix}\ne 0$, i.e. provided that $\widetilde\vg_1$ and $\widetilde\vg_2$ are linearly independent.
The eigenvalues of $\mQ$ are
$$
   \gamma_1^2=\frac{1}{8} \left(\mu
   _1^2+\mu _2^2+\mu _3^2+\mu _4^2+\sqrt{\left(\left(\mu _2+\mu _3\right){}^2+\left(\mu
   _1-\mu _4\right){}^2\right) \left(\left(\mu _2-\mu _3\right){}^2+\left(\mu _1+\mu
   _4\right){}^2\right)}\right),$$
$$\gamma_2^2=\frac{1}{8} \left(\mu _1^2+\mu _2^2+\mu _3^2+\mu _4^2-\sqrt{\left(\left(\mu _2+\mu
   _3\right){}^2+\left(\mu _1-\mu _4\right){}^2\right) \left(\left(\mu _2-\mu
   _3\right){}^2+\left(\mu _1+\mu _4\right){}^2\right)}\right),$$
so condition $\gamma_1>\gamma_2$ is satisfied except when $\mu_1=\pm \mu_4$ and $\mu_2=\mp\mu_3$, i.e. except when
 $\widetilde\vg_1$ and $\widetilde\vg_2$ are orthogonal and of the same length.

We see that $\rho_r\topp N=2\gamma_r\sqrt{MN}$, which determines the constants $c_r=2\gamma_r$ for
\eqref{Warm-up-CLT}. Next, we determine the remaining significant terms in \eqref{eq:la-expand}. First, we check that the
shifts $m_r\topp N$ in \eqref{eq:la-expand} do not depend on $N$. To do  so we use \eqref{mmm-b} with matrices
$$
\widetilde{\mathbf{\Sigma}}_R=\frac{MN}{2}\begin{bmatrix}
\mu_1^2(\sigma_1^2+\sigma_3^2) + \mu_2^2(\sigma_2^2+\sigma_4^2) & \mu_1\mu_3(\sigma_1^2+\sigma_3^2) + \mu_2\mu_4(\sigma_2^2+\sigma_4^2) \\
\mu_1\mu_3(\sigma_1^2+\sigma_3^2) + \mu_2\mu_4(\sigma_2^2+\sigma_4^2) & \mu_3^2(\sigma_1^2+\sigma_3^2) + \mu_4^2(\sigma_2^2+\sigma_4^2) \\
\end{bmatrix}
$$
and
$$
\widetilde{\mathbf{\Sigma}}_S=\frac12\begin{bmatrix}
\sigma_1^2+\sigma_2^2  & 0 \\
0 &  \sigma_3^2+\sigma_4^2  \\
\end{bmatrix},
$$
which follow from
$$[\mD_R]_{j,j}= \begin{cases}
(\sigma_1^2+\sigma_3^2)/2 & j\leq N  \\
 (\sigma_2^2+\sigma_4^2)/2 & N+1\leq j \leq 2N
\end{cases} \mbox{ and }
[\mD_S]_{i,i}=\begin{cases}
  (\sigma_1^2+\sigma_2^2)/2 & i\leq M \\
  (\sigma_3^2+\sigma_4^2)/2& M+1\leq i\leq 2M
\end{cases}
$$

To verify normality of the limit, we show that the matrix $\mZ_0$ is asymptotically centered normal, so formula \eqref{eq:Z_r}
gives a bivariate normal distribution in the limit. Denoting as previously by $X_{i,j}$ the entries of matrix $\mC=\mDD-\E\mDD$,
 \eqref{eq:tildeZ_0} gives

\begin{multline*}
\widetilde\mZ_0=\frac{1}{2\sqrt{MN}}\begin{bmatrix} \mu_1 \displaystyle\sum_{i=1}^M\sum_{j=1}^NX_{i,j} + \mu_2
\sum_{i=1}^M\sum_{j=N+1}^{2N}X_{i,j} & \displaystyle\mu_1 \sum_{i=M+1}^{2M}\sum_{j=1}^{N}X_{i,j} + \mu_2 \sum_{i=M+1}^{2M}\sum_{j=N+1}^{2N}X_{i,j}\\ & & \\
\displaystyle \mu_3 \sum_{i=1}^{M}\sum_{j=1}^NX_{i,j} +
\mu_4\sum_{i=1}^{M}\sum_{j=N+1}^{2N}X_{i,j}
  & \displaystyle\mu_3 \sum_{i=M+1}^{2M}\sum_{j=1}^NX_{i,j} + \mu_4\sum_{i=M+1}^{2M}\sum_{j=N+1}^{2N}X_{i,j} \\
\end{bmatrix}\\
\toD \tfrac12\begin{bmatrix}
  \mu_1\sigma_1 \zeta_1 +  \mu_2\sigma_2  \zeta_2  & \mu_1\sigma_3 \zeta_3 +  \mu_2\sigma_4  \zeta_4\\
  \mu_3\sigma_1 \zeta_1 +  \mu_4\sigma_2 \zeta_2   &  \mu_3\sigma_3 \zeta_3 +  \mu_4\sigma_4  \zeta_4\\
\end{bmatrix}
\end{multline*}
with independent $N(0,1)$ random variables $\zeta_1,\dots,\zeta_4$.

In particular, the limit $(Z_1,Z_2)=(m_1,m_2)+(Z_1^\circ,Z_2^\circ)$ is normal with mean given by \eqref{mmm-b}, and centered
bivariate  normal random variable
$$Z_1^\circ=\tfrac1{2{\gamma_1}} \vv_1^* \widetilde\mZ_0\vv_1, \quad
Z_2^\circ=\tfrac1{2{\gamma_2}} \vv_2^*\widetilde\mZ_0\vv_2.
 $$

\end{proof}

\subsubsection{Numerical example and simulations}\label{Sect:numeric}
For a
numerical example, suppose $\sigma_j^2=\sigma^2$, $\mu_j=\mu$, except $\mu_1=0$. Then $\gamma_1^2= \frac{\mu^2}{8}
\left(3+\sqrt{5}\right), \gamma_2^2=\frac{\mu^2}{8}
   \left(3-\sqrt{5}\right)$ and
   $$\vv_1=\frac{1}{\sqrt{10}}\begin{bmatrix}
 \sqrt{ 5-\sqrt{5}  } \\ \\
     \sqrt{ 5+\sqrt{5} }
   \end{bmatrix}\approx \begin{bmatrix}
     0.525731\\ 0.850651
   \end{bmatrix}$$
$$\vv_2=\frac{1}{\sqrt{10}}\begin{bmatrix}
    \sqrt{ 5+\sqrt{5}}\\ \\
-   \sqrt{5-\sqrt{5}}
   \end{bmatrix}\approx \begin{bmatrix}
  0.850651\\ - 0.525731
   \end{bmatrix}$$

$$\widetilde{\mathbf{\Sigma}}_S=\sigma^2\mI_2, \quad \widetilde{\mathbf{\Sigma}}_R= M N \mu ^2\sigma^2\left[
\begin{array}{cc}
1 &1 \\
1 & 2  \\
\end{array}
\right]$$ so with $c=M/N$, formula \eqref{mmm-b} gives
 $$m_1 =\frac{(\sqrt{5}-1)\sigma^2(M+N)}{2\mu\sqrt{MN}} ,\quad
m_2= \frac{ (\sqrt{5}+1)\sigma^2(M+N)}{2\mu\sqrt{MN}}.$$

We get

$$
Z_1^\circ=\frac{\left(2 \sqrt{5} \zeta_1+\left(5+\sqrt{5}\right)
   \zeta_2+\left(5+\sqrt{5}\right) \zeta_3+\left(5+3 \sqrt{5}\right)
   \zeta_4\right) }{5 \left(1+\sqrt{5}\right)
   }  \sigma ,
$$
$$
Z_2^\circ=\frac{\left(-\left(5+\sqrt{5}\right) \zeta_1+2 \sqrt{5} \zeta_2+2
   \sqrt{5} \zeta_3+\left(\sqrt{5}-5\right) \zeta_4\right)
   }{10 } \sigma.
$$

 Thus $\la_1-\mu\sqrt{\tfrac12MN(3+\sqrt{5})},\la_2-\mu\sqrt{\tfrac12MN(3-\sqrt{5})}$ is approximately normal with mean $(m_1,m_2)$ and
 covariance matrix
 $
\sigma^2\mI_2
 $.
In particular, if the entries of matrices are independent uninform $U(-1,1)$ for block $\mA_1$ and $U(0,2)$ for blocks $\mA_2,\mA_3,\mA_4$, then $\sigma^2=1/3, \mu=1$. So
 with $M=20,N=50$  we get
$$\la_1\approx 51.6228
+\frac1{\sqrt{3}}\zeta_1 ,  \; \la_2\approx 20.7378+\frac{1}{\sqrt{3}}\zeta_2$$ with (new) independent normal N(0,1) random variables
$\zeta_1,\zeta_2$. %
Figure \ref{F1} show the result of simulations for two sets of choices of $M,N$.

\pgfdeclareimage[width=3in]{HL}{warmup-large}
\pgfdeclareimage[width=3in]{HS}{warmup-small}

\begin{figure}[H]
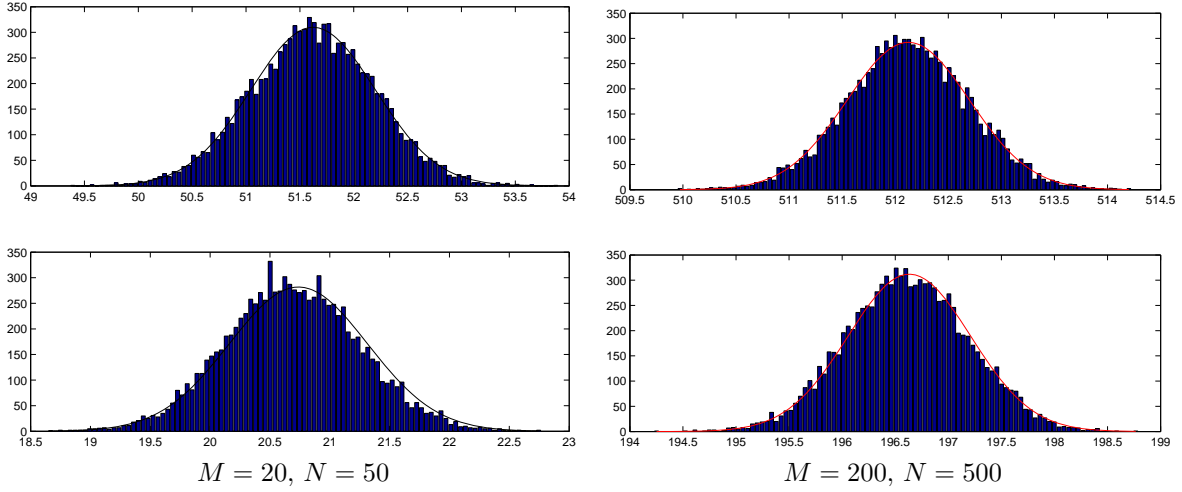

  \begin{center}
     \begin{tabular}{ccc}
      \pgfuseimage{HS} &\pgfuseimage{HL} \\
      $M=20$, $N=50$ & $M=200$, $N=500$
     \end{tabular}
  \end{center}

\caption{Histograms of simulations of 10,000 realizations, overlayed with normal density of variance $1/3$. Top row: Largest singular value; Second row: second singular value.
For small $N$, additional poorly
controlled error arises from  $\eps\topp N\to 0$ in probability. \label{F1}}
\end{figure}

\subsection{Application to a model in population genetics}\label{Sect: SNP}

Following  \cite{bryc2013separation} (see also \cite{Patterson2006}), we consider  an  $M\times N$ array $\mDD$ of genetic markers  with rows labeled by individuals
 and columns labeled by polymorphic markers. The entries $[\mDD]_{i,j}$ are the number of alleles for marker $j$,
 individual $i$, are assumed independent, and take values $0,1,2$ with probabilities $(1-p)^2, 2p(1-p), p^2$ respectively, where $p$ is the frequency of the $j$-th allele.
 We assume that we have data for $M$ individuals from $K$ subpopulation and that we have $M_r$ individuals from the subpopulation labeled $r$.
For our asymptotic analysis where $N\to \infty$ we assume \eqref{eq:M/N} and that each subpopulation is sufficiently represented in the data so that
\begin{equation*}
  \label{eq:Mr/N} M_r/N\to c_r>0,
\end{equation*}
where of course $c_1+\dots+c_K=c$. (Note that our notation for $c_r$ is slightly different than the notation in \cite{bryc2013separation}.)
We assume that allelic frequency for the $j$-th marker depends on the subpopulation of which the individual is a member but does not depend on the individual otherwise.
Thus with the $r$-th subpopulation we associate  the  vector
$\vp_r\in(0,1)^N$ of allelic probabilities, where  $p_r(j):=[\vp_r]_j$ is the value of $p$ for the $j$-th marker, $j=1,2,\dots,N$.

We further assume that the allelic frequencies are fixed, but arise  %
from some regular mechanism,
 which guarantees that for $d=1,2,3,4$ the following limits exist
\begin{equation*}\label{a:p-conv}
\lim_{N\to\infty} \frac1N \sum_{j=1}^N  p_{r_1}(j)p_{r_2}(j)\dots p_{r_d}(j)= \pi_{r_1,r_2,\dots ,r_d},\; 1\leq r_1\leq\dots\leq r_d\leq K.
\end{equation*}
This holds if the allelic probabilities $p_r(j)$ for the
$r$-th population arise in ergodic fashion from  joint allelic spectrum $\varphi(x_1,\dots,x_K)$  \cite{Kimura:1964} with
\begin{equation}\label{spectrum}
\pi_{r_1,r_2,\dots ,r_d}=\int_{[0,1]^K} x_{r_1} \dots x_{r_d}\varphi(x_1,\dots,x_K)d x_1\dots d x_K.
\end{equation}

Under the above assumptions, the entries of $\mDD$ are independent Binomial random variables with the same number of  trials 2, but with varying probabilities of success.
Using the assumed distribution of the entries of $\mDD$ we have $\mB=\E\mDD=2\sum_{r=1}^K \widetilde\ve_r\vp_r^*$, where $\widetilde\ve_r$ is the vector indicating the locations of the members of the
$r$-th subpopulation, i.e. $[\widetilde\ve_r]_i=1$ when the $i$-th individual is a member of the $r$-th subpopulation.
Assuming the entries of $\mDD$ are independent, we get %
$\mB=\sum_{r=1}^K \widetilde\vf_r\widetilde\vg_r^*=\widetilde\mF\widetilde\mG^*$  with orthonormal vectors
$\widetilde\vf_r=\widetilde\ve_r/\sqrt{M_r}$ and with $\widetilde\vg_s=2\sqrt{M_s}\vp_s$, so we have \eqref{B-from-tilde}.
In this setting, Remark \ref{Remark-Q} applies. In \eqref{R0}, we have $[\widetilde\mR_0]_{r,s}=4\sqrt{M_rM_s}\vp_r^*\vp_s$
and   $\widetilde\mR_0/(MN)\to \mQ$, where
\begin{equation}\label{mQ-gen}
[\mQ]_{r,s}:= 4\frac{\sqrt{c_rc_s}}{c}\pi_{r,s},
\end{equation}
so the eigenvalues of $\widetilde\mR_0$ are $\rho_r^2\sim \gamma_r^2  MN +o(N^2)$.
As previously, we assume that $\mQ$
 has positive and distinct eigenvalues $\gamma_1^2>\gamma_2^2>\dots>\gamma_K^2>0$ with corresponding eigenvectors
 $\vv_1,\dots,\vv_k\in\RR^K$.
(Due to change of notation, matrix \eqref{mQ-gen} differs  from \cite[(2.6)]{bryc2013separation}  by a  factor of 4.)

 To state the result, for $1\leq t\leq K$    we introduce matrices ${\mathbf{\Sigma}}_t\in\calM_{K\times K}$ with entries
\begin{equation}\label{Sigma_k}
[{\mathbf{\Sigma}}_t]_{r,s}= \frac{\sqrt{c_rc_s}}{c}(\pi_{r,s,t}-\pi_{r,s,t,t}).
\end{equation}

\begin{proposition}
  \label{T-gen}
The  $K$ largest singular values of $\mDD       $ are approximately normal,
\begin{equation*}
\begin{bmatrix}
  \la_1\topp N-{\rho_1\topp N} \\
\la_2\topp N-{\rho_2\topp N} \\
\vdots \\
\la_K\topp N-{\rho_K\topp N}
\end{bmatrix} \toD \begin{bmatrix}
  m_1\\m_2\\\vdots \\ m_K
\end{bmatrix}+\begin{bmatrix}
  \zeta_1 \\\zeta_2\\\vdots\\\zeta_K
\end{bmatrix}
\end{equation*}
where
\begin{equation}
  \label{m_r}
  m_r= \frac{1}{ \sqrt{c} \gamma_r }\sum_{t=1}^K  [\vv_r]_t^2(\pi_t-\pi_{t,t})+
  \frac{4\sqrt{c}}{\gamma_r^{3}}\sum_{t=1}^K \frac{c_t}{c}\vv_r^*\mathbf{\Sigma}_t \vv_r
\end{equation}
and
$(\zeta_1,\dots,\zeta_K)$ is centered multivariate normal with   covariance
\begin{equation*}\label{zeta-cov}
  \E (\zeta_r\zeta_s)=\frac{8}{ {\gamma_r\gamma_s} } \sum_{t=1}^K [\vv_r]_t[\vv_s]_t\vv_r^*\mathbf{\Sigma}_t\vv_s.
\end{equation*}

\end{proposition}

\begin{proof}
We apply Theorem \ref{Thm:expansion} in the form stated as Remarks \ref{Remark-tilde-all} and \ref{Remark-Q}. The first step is to note that due to the form of vectors $\widetilde\vf_k$,
equation \eqref{eq:tildeZ_0} gives
 a matrix $\widetilde\mZ_0$ with independent columns. Our first task is to show that $\widetilde\mZ_0$ is asymptotically normal by verifying that each  of its independent columns is asymptotically normal.

 Denote  by $\calN_k$ the index set for the $k$-th subpopulation (i.e., $[\widetilde\ve_k]_i=1$ if $i\in\calN_k$).
 In this notation, the $k$-th column of $\widetilde\mZ_0$ is

$$ \frac{2}{\sqrt{MNM_k}}\begin{bmatrix}
 \sqrt{M_1} \sum_{i\in \calN_k} \sum_{j=1}^N X_{i,j} p_1(j) \\
 \sqrt{M_2} \sum_{i\in \calN_k} \sum_{j=1}^N X_{i,j} p_2(j) \\
\vdots \\
 \sqrt{M_K} \sum_{i\in \calN_k} \sum_{j=1}^N X_{i,j} p_K(j)
\end{bmatrix}.
$$

To verify asymptotic normality and find the covariance, we fix $\vt=[t_1,\dots, t_K]^*$. Then the dot product of $\vt$   with the $k$-th column of $\widetilde\mZ_0$ is

\begin{equation*}
  \label{S_N}
  S_N=\sum_{i\in \calN_k} \sum_{j=1}^N a_j(N) X_{i,j}
\end{equation*}
with
$$a_j(N)=\frac{2}{\sqrt{MNM_k}}\sum_{r=1}^K \sqrt{M_r} t_r p_r(j).$$
We first note that by independence
\begin{multline*}
\Var(S_N)=\sum_{i\in \calN_k} \sum_{j=1}^N a_j^2(N) \E X_{i,j}^2 \\ =8 \sum_{r_1,r_2=1}^K t_{r_1}t_{r_2}\frac{\sqrt{M_{r_1}M_{r_2}}}{M}\frac1N\sum_{j=1}^N p_{r_1}(j)p_{r_2}(j) p_k(j)(1-p_k(j))
\\ \to 8 \sum_{r_1,r_2=1}^K \frac{\sqrt{c_{r_1}c_{r_2}}}{c}\left(\pi_{r_1,r_2,k}-\pi_{r_1,r_2,k,k}\right)t_{r_1}t_{r_2}=8\vt^* \mathbf{\Sigma}_k \vt
\end{multline*}
giving the covariance matrix for the $k$-column  as 8 times \eqref{Sigma_k}.

Next we note that since $E(X^4_{i,j})=2p_k(j)(1-p_k(j))$, we have
\begin{multline*}
\sum_{i\in \calN_k} \sum_{j=1}^N a_j^4 (N)\E X_{i,j}^4 \\
 = \frac{32}{N M_k}\sum_{r_1,r_2,r_3,r_4=1}^K \frac{\sqrt{M_{r_1}M_{r_2}M_{r_3}M_{r_4} }}{M^2}t_{r_1}t_{r_2}t_{r_3}t_{r_4}
\frac1N\sum_{j=1}^Np_{r_1}(j)p_{r_2}(j)p_{r_3}(j)p_{r_4}(j)p_k(j)(1-p_k(j))=O(1/N^2)\to 0.
\end{multline*}
By Lyapunov's theorem    $S_N$ is asymptotically normal.
Thus  the $k$-th column of $\widetilde\mZ_0$ is asymptotically normal with covariance  8 times
\eqref{Sigma_k}. Let %
$\mZ_0\topp\infty$ denote the distributional limit of  $\widetilde\mZ_0$.

From \eqref{eq:Z_r-b} with $\widetilde\vu_r$ replaced by $\vv_r$ as in Remark \ref{Remark-Q}, we see that $(Z_1\topp N, \dots, Z_K\topp N)$ converges in distribution to the multivariate
 normal r.v. $ (\zeta_1,\dots,\zeta_K)$ with  covariance %
$$\E(\zeta_r\zeta_s)=\frac{1}{{\gamma_r\gamma_s} } \E\left( \vv_r^*\mZ_0\topp\infty\vv_r\vv_s^*\mZ_0\topp\infty\vv_s\right)
=
\frac{8}{{\gamma_r\gamma_s} }\sum_{t=1}^K [\vv_r]_t[\vv_s]_t\vv_r^*\mathbf{\Sigma}_t\vv_s.$$

Next, we use  formula \eqref{mmm-b} to compute the shift.
We first compute $\widetilde{\mathbf{\Sigma}}_S=\E (\widetilde\mF^*\mC\mC^*\widetilde\mF)/N$.
As already noted,  $\mC^*\widetilde\mF\in\calM_{ N\times K}$ has $K$ independent columns, with the $k$-th column

$$ \frac{1}{\sqrt{M_k}}\begin{bmatrix}
  \sum_{i\in \calN_k}   X_{i,1}   \\
   \sum_{i\in \calN_k}   X_{i,2}  \\
\vdots \\
 \sum_{i\in \calN_k} \ X_{i,N}
\end{bmatrix}.
$$
So  $\widetilde{\mathbf{\Sigma}}_S$ is a diagonal matrix with
$$[\widetilde{\mathbf{\Sigma}}_S]_{rr}=\frac2N\sum_{j=1}^Np_r(j)(1-p_r(j))\to 2(\pi_r-\pi_{rr}).
$$
Next, we compute the limit of
$$\frac{c}{MN}\widetilde{\mathbf{\Sigma}}_R=\frac{c}{M^2N}\E (\widetilde\mG^*\mC^*\mC\widetilde\mG)\sim \frac{1}{M N^2}\E (\widetilde\mG^*\mC^*\mC\widetilde\mG).$$
Since
$$[\E (\widetilde\mG^*\mC^*\mC\widetilde\mG)]_{rs}=2 \sum_{t=1}^K M_t\sum_{j=1}^N [\vg_s]_j[\vg_r]_jp_t(j)(1-p_t(j))=
8 \sqrt{M_rM_s}\sum_{t=1}^K M_t\sum_{j=1}^N p_s(j)p_r(j)p_t(j)(1-p_t(j))$$
we see that
$$\left[\frac{c}{MN}\widetilde{\mathbf{\Sigma}}_R\right]_{rs} \to  \frac{8 \sqrt{c_rc_s}}{c} \sum_{t=1}^K c_t  (\pi_{r,s,t}-\pi_{r,s,t,t}) .$$
This shows that
$$
\frac{c}{MN}\widetilde{\mathbf{\Sigma}}_R  \to 8\sum_{t=1}^K c_t \mathbf{\Sigma}_t.
$$
From \eqref{mmm-b}  we calculate
$$
  m_r=\sum_{s=1}^K\left( \frac{1}{\sqrt{c }\gamma_r} [\vv_r]_s^2(\pi_s-\pi_{s,s})+
  \frac{4 c_s}{c^{3/2}\gamma_r^{3}}\sum_{t_1,t_2=1}^K [\vv_r]_{t_1} [\vv_r]_{t_2} \sqrt{c_{t_1}c_{t_2}} (\pi_{t_1,t_2,s}-\pi_{t_1,t_2,s,s}) \right)
$$
which is \eqref{m_r}.
\end{proof}

Ref. \cite{bryc2013separation} worked  with the eigenvalues of the 
"sample covariance matrix" $(\mDD \mDD^*)/(\sqrt{M}+\sqrt{N})^2$,
i.e., with the normalized squares of singular values
$\La_r=\la_r^2/(\sqrt{M}+\sqrt{N})^2$.
Proposition \ref{T-gen} then gives  the following normal approximation.
\begin{proposition}\label{Prop-separation}
$$\begin{bmatrix}
  \La_1-\frac{(\rho_1\topp N)^2}{(\sqrt{M}+\sqrt{N})^2} \\ \\
  \La_2-\frac{(\rho_2\topp N)^2}{(\sqrt{M}+\sqrt{N})^2} \\ \\
  \vdots \\
  \La_K-\frac{(\rho_K\topp N)^2}{(\sqrt{M}+\sqrt{N})^2} \\
\end{bmatrix} \toD
\begin{bmatrix}
  \tilde m_1\\ \tilde m_2\\\vdots \\\tilde  m_K
\end{bmatrix}+\begin{bmatrix}
  Z_1 \\Z_2\\\vdots\\Z_K
\end{bmatrix}
$$
with recalculated shift
\begin{equation*}
  \label{m_r+}
  \tilde m_r= \frac{2}{ (1+c^{1/2})^2}\sum_{t=1}^K  [\vv_r]_t^2(\pi_t-\pi_{t,t})+
  \frac{8c }{\gamma_r^2 (1+c^{1/2})^2 }\sum_{t=1}^K \frac{c_t}{c}\vv_r^*\mathbf{\Sigma}_t \vv_r
\end{equation*}
and with recalculated centered  multivariate normal {random vector}
$(Z_1,\dots,Z_K)$  with  covariance
\begin{equation*}\label{zeta-cov+}
  \E (Z_rZ_s)=\frac{32c}{(1+\sqrt{c})^4}  \sum_{t=1}^K [\vv_r]_t[\vv_s]_t\vv_r^*\mathbf{\Sigma}_t\vv_s.
\end{equation*}
\end{proposition}
\begin{proof}
 $$\La_r-\frac{\rho_r^2}{(\sqrt{M}+\sqrt{N})^2}=(\la_r- {\rho_r})\times \frac{\la_r+ {\rho_r}}{\sqrt{MN}}\times
\frac{\sqrt{MN}}{(\sqrt{M}+\sqrt{N})^2}.$$
Since
$$\frac{\la_r+ {\rho_r}}{\sqrt{MN}}\to 2 {\gamma_r} \mbox{ in probability,  and }  \frac{\sqrt{MN}}{(\sqrt{M}+\sqrt{N})^2}\to
\frac{\sqrt{c}}{(1+\sqrt{c})^2},
$$
see Lemma \ref{L:P(Omega)}, the result follows.
\end{proof}

\subsubsection{Numerical illustration}\label{Sec:NumIll}
As an illustration of Theorem \ref{T-gen}, we re-analyze the example from \cite[Section 3.1]{bryc2013separation}. In that example,
the subpopulation sample sizes were drawn with proportions
$c_1=c/6$, $c_2=c/3$, $c_3=c/2$ where $c=M/N$ varied from case to case. The theoretical population proportions $p_r(j)$ at  each location
for each subpopulation were selected from the same  site frequency spectrum
   $\psi(x)=1/(2\sqrt{x})$.
Following  \cite[Section 3.1]{bryc2013separation}, for our simulations we  selected $p_1(j), p_2(j),p_3(j)$ independently at each location $j$,
which corresponds to joint allelic spectrum  $\psi(x,y,z)=\psi(x)\psi(y)\psi(z)=1/(8\sqrt{xyz})$  in \eqref{spectrum}.

In this setting, we can explicitly compute the theoretical matrix of moments \eqref{spectrum}
     and   matrix  $\mQ$  defined by \eqref{Assume-q-lim}: %
$$
[\pi_{r,s}]=\left[
\begin{array}{ccc}
 {1}/{5} & {1}/{9} & {1}/{9} \\
 {1}/{9} & {1}/{5} & {1}/{9} \\
 {1}/{9} & {1}/{9} & {1}/{5} \\
\end{array}\right],\;
\mQ=\frac{4}{c}\left[\sqrt{c_rc_s}\,\pi_{r,s}\right]=
\left[
\begin{array}{ccc}
 \frac{2}{15} & \frac{2 \sqrt{2}}{27} & \frac{2}{9 \sqrt{3}} \\  \\
 \frac{2 \sqrt{2}}{27} & \frac{4}{15} & \frac{2 \sqrt{2} }{9\sqrt{3}} \\    \\
 \frac{2}{9 \sqrt{3}} & \frac{2 \sqrt{2} }{9\sqrt{3}} & \frac{2}{5} \\
\end{array}
\right]
=
\left[
\begin{array}{ccc}
 0.133333 & 0.104757 & 0.1283 \\
 0.104757 & 0.266667 & 0.181444 \\
 0.1283 & 0.181444 & 0.4 \\
\end{array}
\right]
$$
(Due to change of notation, this  matrix and the eigenvalues are  $4$ times    the corresponding values from \cite[page 37]{bryc2013separation}.)
The eigenvalues of the above matrix    $\mQ$ are
$$[\gamma_1^2,\gamma_2^2,\gamma_3^2]= [ 0.586836, 0.141985, 0.0711794]$$ and the corresponding eigenvectors are
$$
\vv_1= \begin{bmatrix}
 0.342425 \\
 0.545539 \\
 0.764939
\end{bmatrix},\; \vv_2=\begin{bmatrix}
 -0.154523 \\
 -0.770372 \\
 0.618586
\end{bmatrix},\; \vv_3=\begin{bmatrix}
 -0.926751 \\
 0.33002 \\
 0.179496
\end{bmatrix}.
$$
In order to apply the formulas, we use \eqref{spectrum} to compute
$$\pi_{r,s,t}=\begin{cases}
1/27 & r\ne s \ne t \\
1/15 & \mbox{ if one pair of indexes is repeated}\\
 1/7 & r=s=t
\end{cases}$$
and
$$\pi_{r,s,t,t}=\begin{cases}
1/45 & r\ne s \ne t \\
1/25 &r=s\ne t \\
1/21 & r\ne s=t \\
 1/9 & r=s=t
\end{cases}$$
As an intermediate step towards \eqref{Sigma_k}, it is convenient to collect the above data into three auxiliary matrices
$$[\pi_{rs1}-\pi_{rs11}]_{r,s}=\left[
\begin{array}{ccc}
 \frac{2}{63} & \frac{2}{105} & \frac{2}{105} \\  \\
 \frac{2}{105} & \frac{2}{75} & \frac{2}{135} \\    \\
 \frac{2}{105} & \frac{2}{135} & \frac{2}{75} \\
\end{array}
\right]$$
$$[\pi_{rs2}-\pi_{rs22}]_{r,s}=\left[
\begin{array}{ccc}
 \frac{2}{75} & \frac{2}{105} & \frac{2}{135} \\  \\
 \frac{2}{105} & \frac{2}{63} & \frac{2}{105} \\    \\
 \frac{2}{135} & \frac{2}{105} & \frac{2}{75} \\
\end{array}
\right]$$
$$[\pi_{rs3}-\pi_{rs33}]_{r,s}=
\left[
\begin{array}{ccc}
 \frac{2}{75} & \frac{2}{135} & \frac{2}{105} \\   \\
 \frac{2}{135} & \frac{2}{75} & \frac{2}{105} \\  \\
 \frac{2}{105} & \frac{2}{105} & \frac{2}{63} \\
\end{array}
\right]$$
From \eqref{Sigma_k} we get
$$
  \mathbf{\Sigma}_1=\left[
\begin{array}{ccc}
 \frac{1}{189} & \frac{\sqrt{2}}{315} & \frac{1}{105 \sqrt{3}} \\
 \frac{\sqrt{2}}{315} & \frac{2}{225} & \frac{\sqrt{2}}{135\sqrt{3}} \\
 \frac{1}{105 \sqrt{3}} & \frac{\sqrt{2}}{135\sqrt{3}}& \frac{1}{75} \\
\end{array}
\right]=\left[
\begin{array}{ccc}
 0.00529101 & 0.00448957 & 0.00549857 \\
 0.00448957 & 0.00888889 & 0.00604812 \\
 0.00549857 & 0.00604812 & 0.0133333 \\
\end{array}
\right]
$$
$$
   \mathbf{\Sigma}_2=\left[
\begin{array}{ccc}
 \frac{1}{189} & \frac{\sqrt{2}}{315} & \frac{1}{105 \sqrt{3}} \\
 \frac{\sqrt{2}}{315} & \frac{2}{225} & \frac{\sqrt{2}}{135\sqrt{3}}\\
 \frac{1}{105 \sqrt{3}} & \frac{\sqrt{2}}{135\sqrt{3}} & \frac{1}{75} \\
\end{array}
\right]=\left[
\begin{array}{ccc}
 0.00444444 & 0.00448957 & 0.00427667 \\
 0.00448957 & 0.010582 & 0.00777616 \\
 0.00427667 & 0.00777616 & 0.0133333 \\
\end{array}
\right]
  $$
  $$
    \mathbf{\Sigma}_3=\left[
\begin{array}{ccc}
 \frac{1}{225} & \frac{\sqrt{2}}{405} & \frac{1}{105 \sqrt{3}} \\
 \frac{\sqrt{2}}{405} & \frac{2}{225} &  \frac{\sqrt{2}}{105\sqrt{3}} \\
 \frac{1}{105 \sqrt{3}} & \frac{\sqrt{2}}{105\sqrt{3}} & \frac{1}{63} \\
\end{array}
\right]=\left[
\begin{array}{ccc}
 0.00444444 & 0.00349189 & 0.00549857 \\
 0.00349189 & 0.00888889 & 0.00777616 \\
 0.00549857 & 0.00777616 & 0.015873 \\
\end{array}
\right]
$$
Using these expressions and \eqref{m_r} with $c=M/N=120/2500$, we determine
$$
[m_1,m_2,m_3]=
[0.183948 \sqrt{c}+\frac{0.174053}{\sqrt{c}},0.323598 \sqrt{c}+\frac{0.353849}{\sqrt{c}},0.47449 \sqrt{c}+\frac{0.49976}{\sqrt{c}}]=[0.834739, 1.68599, 2.38504].
$$
We note that the shift is stronger for smaller singular values and is more pronounced for rectangular matrices  with  small (or large) $c$.

Finally, we compute the covariance matrix
\begin{equation*}
  \label{Cov-Th}
  [\E\zeta_r\zeta_s]=\left[
\begin{array}{ccc}
 0.306317 & 0.0293619 & 0.0225604 \\
 0.0293619 & 0.233577 & -0.00941692 \\
 0.0225604 & -0.00941692 & 0.235559 \\
\end{array}
\right].
\end{equation*}
The following figure illustrates that normal approximation works well.

\pgfdeclareimage[width=5in]{BBS}{BBS3.1}
\begin{figure}[H]
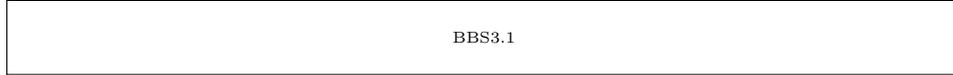

\pgfuseimage{BBS}
\caption{ Histograms of three largest singular values
 with normal curves centered at
$[{\rho_1}+m_1, {\rho_2}+m_2, {\rho_3}+m_3]=[413.47,  208.02,  145.44]$
 with variances $0.3063,    0.2336,    0.2356$.
 Based on 10,000 runs of simulations after a single randomization to choose  vectors $\vp_1,\vp_2,\vp_3$ with $M=120, N=2500$. \label{F2}
}

\end{figure}

Based on simulations reported in Fig. \ref{F2}, we conclude that normal approximation which uses the singular values of $\widetilde\mR_0$ provides  a reasonable fit.

Let us now turn to the question of  asymptotic normality for the normalized squares of singular values
  $\La_r=\la_r^2/(\sqrt{M}+\sqrt{N})^2$.
 With $M=120, N=2500$ (correcting the misreported values)
Ref.
 \cite{bryc2013separation} reported the
 observed values
 $ (  48.2,   11.5,    5.8 )$  for  $\La_1,\La_2,\La_3$ versus the
 ``theoretical estimates" $(  47.4, 11.5, 5.7 )$ calculated as
 $ \gamma_r ^2M N/(\sqrt{M}+\sqrt{N})^2$ in our notation. While the numerical differences are small,
 Proposition \ref{Prop-separation} restricts
  the accuracy of such estimates due to their dependence on  the eigenvalues  of $\widetilde\mR_0$  constructed from
   vectors of   allelic probabilities $\vp_r$.
 Figure \ref{F3E} illustrates  that
 different
selections of such vectors from the same joint allelic spectrum  \eqref{spectrum}
may yield quite different ranges for $\La_1,\La_2,\La_3$.
It is perhaps worth pointing out that {different choices of allelic probabilities affect only the centering;} the variance of normal approximation depends only on the allelic spectrum.

  \pgfdeclareimage[width=3in]{Appl1}{SNP_appl1}
\pgfdeclareimage[width=3in]{Appl2}{SNP_appl2}

\begin{figure}[H]
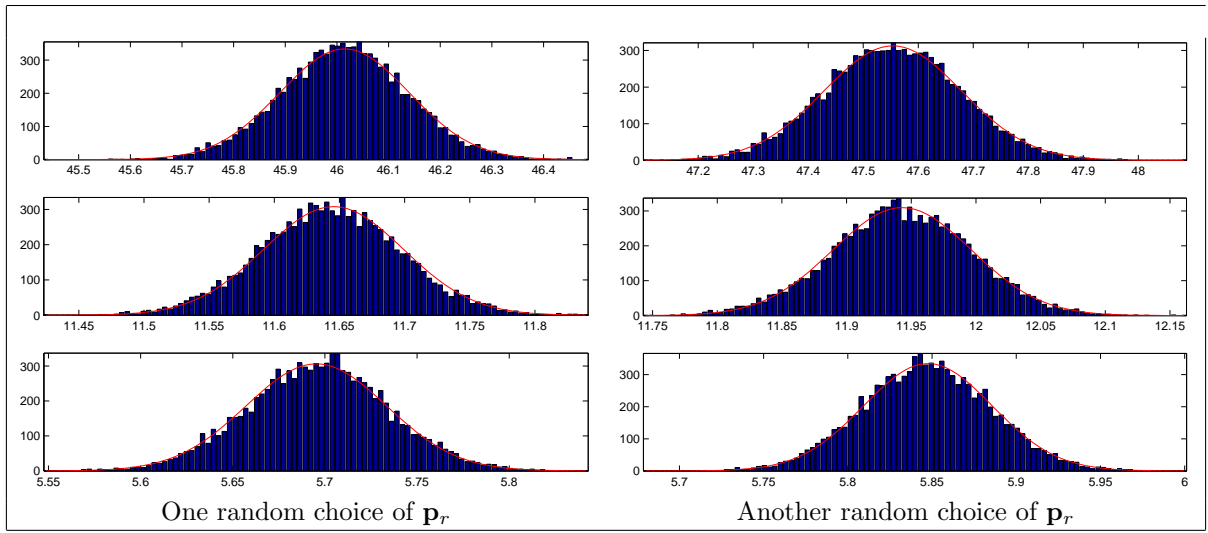

\begin{tabular}{|cc|}
\hline \\
\pgfuseimage{Appl1} &     \pgfuseimage{Appl2}  \\
 One random choice of $\vp_r$  & Another random choice of $\vp_r$   \\ \hline
\end{tabular}
\caption{Two  histograms of normalized squared singular values $(\La_1,\La_2,\La_3)$, based on 10000 simulations,
 and the theoretical normal curves from Proposition \ref{Prop-separation} drawn in red.
  This is $M=120$ individuals with $N=2500$ markers.
  Although the numerical differences between {$\La_1,\La_2,\La_3$ on} the left-hand-side and on the right-hand side are small,
  the histograms have practically disjoint supports.
  \label{F3E}}

\end{figure}

\subsection*{Acknowledgement} {
We thank an anonymous referee for suggestions that lead to a stronger form of Theorem 1.1.
}


\end{document}